\documentclass{article}

\usepackage[utf8]{inputenc}
\usepackage{amssymb, amsmath, amsthm}
\usepackage{mathabx}
\usepackage{hyperref}

\newtheorem{lemma}{Lemma}[section]
\newtheorem{remark}[lemma]{Remark}
\newtheorem{proposition}[lemma]{Proposition}
\newtheorem{theorem}[lemma]{Theorem}
\newtheorem{corollary}[lemma]{Corollary}
\theoremstyle{remark}
\newtheorem{example}[lemma]{Example}

\DeclareMathOperator*{\argmin}{argmin}
\DeclareMathOperator{\sign}{sign}
\newcommand{\real}{\mathbb{R}}

\newcommand{\dualp}[1]{\left\langle #1 \right\rangle} 
\newcommand{\E}{\mathbb{E}}

\begin{document}

\title{A Relaxation Argument for Optimization in Neural Networks and Non-Convex Compressed Sensing}
\author{G. Welper\footnote{Department of Mathematics, University of Central Florida, Orlando, FL 32816, USA, email \href{mailto:gerrit.welper@ucf.edu}{\texttt{gerrit.welper@ucf.edu}}. 
}}

\date{}
\maketitle

\begin{abstract}
  It has been observed in practical applications and in theoretical analysis that over-parametrization helps to find good minima in neural network training. Similarly, in this article we study widening and deepening neural networks by a relaxation argument so that the enlarged networks are rich enough to run $r$ copies of parts of the original network in parallel, without necessarily achieving zero training error as in over-parametrized scenarios. The partial copies can be combined in $r^\theta$ possible ways for layer width $\theta$. Therefore, the enlarged networks can potentially achieve the best training error of $r^\theta$ random initializations, but it is not immediately clear if this can be realized via gradient descent or similar training methods.

  The same construction can be applied to other optimization problems by introducing a similar layered structure. We apply this idea to non-convex compressed sensing, where we show that in some scenarios we can realize the $r^\theta$ times increased chance to obtain a global optimum by solving a convex optimization problem of dimension $r\theta$.
\end{abstract}



\section{Introduction}

Neural networks are trained with gradient descent or related methods starting from random initial values. Since the loss function is non-convex, this can in principle result in bad local minima. Indeed, in the worst case, the problem of neural network training is $NP$-hard \cite{BlumRivest1989} and this behavior can be expected. Nonetheless, neural networks are successfully trained in a large number of practical applications \cite{GoodfellowBengioCourville2016,HeZhangRenEtAl2016}. Contrary to arbitrary networks in worst case scenarios, practical networks are usually over-parametrized, which has been studied experimentally in e.g. \cite{GoodfellowVinyals2015,ZhangBengioHardtEtAl2017}. On the theoretical side, over-parametrization usually means that the networks are powerful enough to achieve zero training error in which case convergence guarantees of gradient descent methods are available \cite{SoudryCarmon2016,SafranShamir2018,LiLiang2018,Allen-ZhuLiSong2019,DuLeeLiEtAl2019}.

In this article, we study the benefits of adding extra weights to neural networks and other non-convex optimization problems like $\ell_p$-minimization with $p<1$ in compressed sensing, without necessarily being over-parametrized. To this end, we start with a non-convex reference problem and enlarge it by a relaxation argument that is motivated by the process of widening and deepening a neural network. The enlarged network does not necessarily achieve zero training error, but has sufficient capacity to run several (say $r$), instances of parts of the reference problem in parallel together with some extra \emph{selector variables}. By combining the parallel pieces via a proper choice of the selector variables, the enlarged network can be reduced to the reference network in $r^\theta$ different ways, where $\theta$ is the width of original network.

Therefore, if we can compute the optimal selector variables, we can potentially find a minimizer that is comparable to the best of $r^\theta$ numerical optima from random initializations of the reference problem. However, the enlarged network only has the capacity to run $r$ of those combinations in parallel. On the one hand, that allows us to run the network much faster than the exponential number $r^\theta$ of combinations, but on the other hand it is not clear to what extend a gradient descent method can realize the described gains. The relaxation argument can be applied to several non-convex optimization problems and therefore, in this paper we analyze the method for compressed sensing, which is better understood than neural network training.

In compressed sensing, we search for the sparsest solution of an under-determined linear system, i.e. for a measurement matrix $A \in \real^{m \times N}$ and measurements $y \in \real^m$, we are interested in the solution of the optimization problem
\begin{equation}
  \begin{aligned}
     & \min_{x \in \real^N} \|x\|_0, & & \text{s.t.} & Ax & = y,
  \end{aligned}
  \label{eq:cs0}
\end{equation}
where the $\ell_0$-norm measures the number of non-zero entries. Since this problem is computationally difficult, it is typically replaced by a $\ell_p$-minimization
\begin{equation}
  \begin{aligned}
    & \min_{x \in \real^N} \|x\|_p^p, & & \text{s.t.} & Ax & = y,
  \end{aligned}
  \label{eq:cs}
\end{equation}
with $0 < p \le 1$. The most common choice is $p=1$, for which the optimization problem is convex and the restricted isometry property (RIP) or similar conditions on the matrix $A$ guarantee that the solutions of the problems \eqref{eq:cs0} and \eqref{eq:cs} coincide, see e.g. \cite{CandesRombergTao2006a,Donoho2006,CandesRombergTao2006,FoucartRauhut2013}. Nonetheless, finding sparse solutions is also of interest in many applications where the RIP is not available. For $p<1$, the $\ell_p$ norm resembles the $\ell_0$ norm more closely and one may expect better sparse recovery results with less assumptions on the matrix $A$. Such results have been reported by several authors \cite{CandesWakinBoyd2008,ChartrandStaneva2008,FoucartLai2009,Sun2012,ShenLi2012}.

For $p$ strictly smaller than one, the optimization problem \eqref{eq:cs} is no longer convex making its optimization considerably more difficult. In fact, in the worst case the problem is NP-hard \cite{Natarajan1995,GeJiangYe2011}. Nonetheless, there are several iterative algorithms \cite{CandesWakinBoyd2008,ChartrandWotaoYin2008,FoucartLai2009,DaubechiesDeVoreFornasierEtAl2010,LaiXuYin2013,WoodworthChartrand2016}, typically variations of reweighted least squares methods, that show promising performance on these problems. Due to the non-convex nature of the problem, the corresponding analysis requires additional assumptions that are hard to validate practically to provide convergence guarantees. 

To this end, we apply the relaxation strategy from the neural network motivation above to the non-convex compressed sensing problem \eqref{eq:cs} to obtain an enlarged problem. This can in principle be solved by reweighted least squares methods, but we do not pursue this in this paper. Instead, to obtain some first insight into the potential of the relaxation argument, as opposed to providing a practical method, we only consider the simpler optimization of the selector variables only. The remaining variables of the original compressed sensing problem are restricted to a discrete set, which is, however, still rich enough to render the problem $NP$-hard, in the worst case. This allows us to analyze a simplified method of randomly guessing the solution in the $r$ copies of the enlarged problem and then finding a good combination by optimizing the selector variables. We show that this achieves the $r^\theta$ fold increased chance to find the global optimum as described in the motivation above, with much weaker assumptions on the sensing matrix $A$ than the RIP. Contrary to running $r^\theta$ separate trials on the reference problem to achieve a similar gain, the enlarged problems requires us to solve a $r \theta$ dimensional convex optimization problem.

The paper is organized as follows. In Section \ref{sec:relaxation}, we state the generic relaxation method, its application to both compressed sensing and neural network training and provide some estimates of potential success probabilities. In Section \ref{sec:cs-relaxation}, we consider the relaxation method for compressed sensing more carefully and prove the main results of the paper.

\section{A Relaxation Method}
\label{sec:relaxation}

In Section \ref{sec:simple-relaxation}, we describe a simple relaxation method and in Section \ref{sec:block-relaxation} a variant with added structure. A discussion of the optimization problems and success probabilities is given in Sections \ref{sec:relaxation-notes} and \ref{sec:relaxation-probabilities}, respectively.

\subsection{Simple Relaxation}
\label{sec:simple-relaxation}

We consider the optimization problem
\begin{equation}
  \begin{aligned}
    & \min_{x \in \real^N} g(x), & 
    & \text{s.t.} &
    h(x) & = 0,
  \end{aligned}
  \label{eq:opt-problem}
\end{equation}
with objective $g$ and constraint $h$ and solve it with a local search method, e.g. gradient descent or variants thereof for neural networks or reweighted least squares for compressed sensing. Since we are particularly interested in non-convex problems, depending on the initial value, this may or may not result in a satisfactory minimizer. Probably the simplest idea to enhance our chance of success is to repeat this optimization for multiple initial values, say $x_k$, $k=1, \dots, r$ resulting in local (numerical) optima $\bar{x}_r$ from which we select the best one
\begin{equation}
  x = \argmin_{k=1, \dots, r} g(\bar{x}_k).
  \label{eq:multiple-initial}
\end{equation}
For simplicity, we drop the equality constraint during this motivation, but all arguments work with it unchanged. In order to relax this problem to a continuous one, note that with standard unit basis vectors $e_k \in \real^r$ and splitting $g(x) = \ell(f(x))$, we can equivalently minimize
\begin{equation*}
  \min_{z \in e_1, \dots, e_r} \ell \left[ \sum_{k=1}^r z_k f(\bar{x}_k) \right].
\end{equation*}
The vector $z$ serves as a ``\emph{selector}'' and picks one guess $f(\bar{x}_k)$ and the split of the objective $g$ into the two components $\ell$ and $f$ allows some flexibility in the placement of the selector. In hope to simplify the problem, we remove the discrete constraint $z \in e_1, \dots, e_r$ in favor of a continuous one $z \in \real^r$ and obtain the \emph{relaxed} problem
\begin{equation}
  \min_{z \in \real^r} \ell \left[ \sum_{k=1}^r z_k f(\bar{x}_k) \right].
  \label{eq:simple_relaxation}
\end{equation}
Similar relaxation strategies are common for many optimization problems, see e.g. \cite{BoydVandenberghe2004,Nesterov2018} in general, \cite{ConfortiCornuejolsZambelli2014} for integer programming or \cite{Villani2003} for optimal transport.  Since the relaxed problem allows a larger choice of selectors $z$, its minimum is at least as small as the un-relaxed one
\begin{equation*}
  \min_{z \in \real^r} \ell \left[ \sum_{k=1}^r z_k f(\bar{x}_k) \right]
  \le
  \min_{k=1, \dots, r} \ell[f(\bar{x}_k)].
\end{equation*}
As a last step, we reintroduce the optimization of the $x$ variable and obtain
\begin{equation}
  \min_{\substack{z \in \real^r \\ x \in \real^N}} \ell \left[ \sum_{k=1}^r z_k f(x_k) \right].
  \label{eq:simple-relaxation}
\end{equation}
A discussion of this problem is given in Section \ref{sec:relaxation-notes}, but before we consider a variant with some additional structure.

\subsection{Block Relaxation}
\label{sec:block-relaxation}

Both the compressed sensing and neural network applications admit extra structure, which we may exploit to write down alternative relaxations of the initial values. Specifically, we may split $f$ and $x$ into blocks $f(x) = [f^1(x^1), \dots, f^\theta(x^\theta)]$ and optimize
\begin{equation}
  \min_{x^1, \dots, x^\theta} \ell \left[ f^1(x^1), \dots, f^\theta(x^\theta) \right],
  \label{eq:opt-structure}
\end{equation}
with each block $f^l(x^l)$ only depending on $x^l$ and not any other $x^j$ with $j \ne l$. The following two examples describe both applications in more detail.

\begin{example}
\label{example:lp}

For $0 < p \le 1$, a matrix $A \in \real^{m \times N}$ and vector $y \in \real^m$ consider the $\ell_p$-minimization
\begin{equation*}
  \begin{aligned}
    & \min_{x \in \real^N} \|x\|_p^p, &
    & \text{s.t.}
    & Ax & = y.
  \end{aligned}
\end{equation*}
If we split $x$ into blocks $x = [x^1, \dots x^\theta]$ with $x^l \in \real^n$ and $n\theta = N$ and define
\begin{align*}
  f^l(x^l) & = \|x^l\|_p^p, &
  \ell(f^1, \dots, f^\theta) = \sum_{l=1}^\theta f^l
\end{align*}
this problem fits into the general structure \eqref{eq:opt-structure}. It has an constraint $h(x) = Ax - y = 0$, which does not influence the discussed relaxation and will be considered more carefully in Section \ref{sec:cs-block-relaxation} below. 

\end{example}

\begin{example}
\label{example:nn}

For neural network training to comply with the structure \eqref{eq:opt-structure} that $f^l$ depends only on one block $x^l$, we consider the relaxation of one layer only and define $x$ as the corresponding weights ($x$ are not the network inputs to comply with the compressed sensing notation). The other layers can be relaxed analogously or optimized alongside of $x$ but are neglected for simplicity in the following motivation.

We split the neural network as $f_{out}(\phi(x f_{in}))$, where $f_{in} \in \real^n$ is the output of previous layers, $x \in \real^{\theta \times n}$ are the weights of the layer we consider, $\phi$ is the element-wise activation function and $f_{out}(\cdot)$ are all downstream layers. In order to fit into the framework \eqref{eq:opt-structure}, we can naturally define the blocks $x^l \in \real^n$ as the rows of $x$ and minimize, a loss function with labels $y$ by
\[
  \operatorname{loss}[y, f_{out}(\phi(x^1 \cdot f_{in}), \dots, \phi(x^\theta \cdot f_{in}))].
\]
This problem fits precisely into the model problem \eqref{eq:opt-structure} with
\begin{align}
  f^l(x^l) & = \phi(x^l \cdot f_{in}), &
  \ell(f^1, \dots f^\theta) & = \operatorname{loss}[y, f_{out}(f^1, \dots, f^\theta)].
  \label{eq:nn-example}
\end{align}
\end{example}

Back to the general problem \eqref{eq:opt-structure}, in order to exploit the extra block structure, we use the same relaxation argument as before.  We start with initial blocks $x^l_k$, $k=1, \dots, r$ and select the best of the resulting numerical minima $[\bar{x}^1_k, \dots, \bar{x}^\theta_k]$ by 
\begin{equation*}
  \min_{k} \ell \left[ f^1(\bar{x}^1_k), \dots, f^\theta(\bar{x}^\theta_k) \right].
\end{equation*}
However, due to the block structure, we can also explore a much bigger search space 
\begin{equation}
  \min_{k^1, \dots k^\theta} \ell \left[ f^1(x^1_{k^1}), \dots, f^\theta(x^\theta_{k^\theta}) \right],
  \label{eq:opt-block}
\end{equation}
which allows us to combine different initial values for every block $f^l(x^l)$ and therefore has a much high chance to include a good initial value. This comes at least with two problems. The first is that we can no longer use the local minimizers $\bar{x}^l_k$. The reason is that the loss $\ell$ couples all blocks so that each optimization process $x^l_k \to \bar{x}^l_k$ not only depends on the initial value $x^l_k$, but also on all other initial values $x^j_k$ with $j \ne l$. Since we now search through all combinations, it is no longer clear how to define $\bar{x}^l_k$. Therefore, we skip this initial optimization for now, but reintroduce it later for the relaxed variant. The second problem is that practically the selection of the optimal $k^1, \dots, k^\theta$ can be very costly, because there are $r^\theta$ possible combinations.

Anyways, let us apply the relaxation argument. We first rewrite the selection as a linear combination
\begin{equation}
  \min_{\substack{z^l \in \{e_1, \dots, e_r)\} \\ l=1, \dots, \theta}} \ell \left[ \sum_{k=1}^r z^1_k f^1(x^1_k), \dots, \sum_{k=1}^r z^\theta_k f^\theta(x^\theta_k) \right],
  \label{eq:block-unrelaxed}
\end{equation}
and then relax it to continuous selectors $z^l$
\begin{equation}
  \min_{z^l \in \real^r, \, l=1, \dots, \theta} \ell \left[ \sum_{k=1}^r z^1_k f^1(x^1_k), \dots, \sum_{k=1}^r z^\theta_k f^\theta(x^\theta_k) \right].
  \label{eq:block-selector}
\end{equation}
Unlike the block-wise selection \eqref{eq:opt-block} of guesses, the relaxed variant \eqref{eq:block-selector} also allows us to reintroduce the optimization of the initial guesses $x^l_k \to \bar{x}^l_k$ by including them in the optimization
\begin{equation}
  \min_{\substack{z^l \in \real^r, \, l=1, \dots, \theta \\ x^l_k, \, k=1, \dots, r, \, l = 1, \dots, \theta}} \ell \left[ \sum_{k=1}^r z^1_k f^1(x^1_k), \dots, \sum_{k=1}^r z^\theta_k f^\theta(x^\theta_k) \right].
  \label{eq:block-relaxation}
\end{equation}
We may also consider other variables in the optimization such as weights from layers that have been neglected in Example \ref{example:nn}.

\subsection{Notes on the optimization problems}
\label{sec:relaxation-notes}

Both relaxed problems \eqref{eq:simple-relaxation} and \eqref{eq:block-relaxation} can be written in the form
\[
  \min_{x_1, \dots, x_r, z} G(x_1, \dots, x_r, z)
\]
with different choices of $G$ and dimensions of $z$. First note that we can choose special values $z_j$ of the selector $z$ so that $G(x_1, \dots, x_r, z_l) = g(x_l)$. Therefore we directly have
\begin{align*}
  \min_{x_1, \dots, x_r, z} G(x_1, \dots, x_r, z) & \le g(\bar{x}_j), &
  j & = 1, \dots, r,
\end{align*}
where as before $\bar{x}_j$ are numerical local optimizers of $g$ with initial values $x_j$. Of course to obtain a fair comparison, we also need a numerical solution of the left hand side. Let $\tilde{x}_i$ and $\tilde{z}$ be such numerical optimizers with the same initial values $x_i$. What can be said about 
\begin{align}
  G(\tilde{x}_1, \dots, \tilde{x}_r, \tilde{z}) & \stackrel{?}{\lesseqqgtr} g(\bar{x}_j), &
  j & = 1, \dots, r?
  \label{eq:main-question}
\end{align}
Let us first make some simple observations:

\begin{enumerate}

  \item In general $\bar{x}_i \ne \tilde{x}_j$ for all $i,j$. The implications depend on the problem at hand. E.g. for neural networks this is fine if the relaxed problem provides better optima and maintains good generalization errors.

  \item The relaxed problem computes all $f(x_j)$, $j=1, \dots, r$ in parallel, but never computes the full outputs $\ell(f(x_j))$ for all $j$. Instead it computes $\ell(\cdot)$ of one mixture of the available $f(x_j)$. Therefore, depending on $\ell$, it is not clear to what extend a gradient descent method can steer the selector variable $z$ to a good choice or balance of the available $f(x_j)$.

  \item Expanding on the last observation, the block relaxation \eqref{eq:block-relaxation} never computes $f([x^1_{k_1}, \cdots, x^\theta_{k_\theta}]) = [f^1(x^1_{k_1}), \dots, f^\theta(x^\theta_{k_\theta})]$ for all $r^\theta$ combinations $k_1, \dots, k_\theta$. This is essential for the runtime since $r^\theta$ quickly becomes prohibitively large, but it raises further questions if gradient descent or similar methods can find the right combination or a good balance.
\end{enumerate}

In summary, the relaxation can significantly reduce the optimization time by avoiding to test an exponential number $r^\theta$ of combinations, but we have to answer the question when it can possibly succeed in finding superior optima.

Some hope comes from our original motivation from deep learning, where it has been observed that larger networks often perform better than smaller ones, see e.g. \cite{GoodfellowVinyals2015,ZhangBengioHardtEtAl2017}. Also several analytical results \cite{SoudryCarmon2016,SafranShamir2018,LiLiang2018,Allen-ZhuLiSong2019,DuLeeLiEtAl2019} show that over-parametrization helps neural network training. These papers usually assume that the networks are rich enough to achieve zero training error. This is not necessarily the case for the relaxation method described above, however the idea is related: We increase the number of network weights and layers in the hope to enable the optimization algorithms to find better minima. This idea is made more concrete in Example \ref{example:nn-relaxation} below. 

\begin{example}
  \label{example:nn-relaxation}

  For the neural network Example \ref{example:nn} the relaxation strategies can be interpreted as follows. We start by making $r$ independent copies of the weights $X^T = [x^1, \cdots, x^\theta] \in \real^{n \times \theta}$ to obtain the new weights $X_k^T = [x^1_k, \dots, x^\theta_k] \in \real^{n \times \theta}$ for $k=1, \dots, r$, thereby effectively widening or ``over-parametrizing'' the layer $\phi(X f_{in})$ from $\real^\theta$ to $\real^{\theta r}$ giving the new hidden layer
  \begin{align*}
    h & := [\phi(X_1 f_{in}), \dots, \phi(X_r f_{in})] \\
    & = [\phi(x^1_1 \cdot f_{in}), \dots, \phi(x^\theta_1 \cdot f_{in}), \cdots, \phi(x^1_r \cdot f_{in}), \dots, \phi(x^\theta_r \cdot f_{in})] \in \real^\theta \otimes \real^r.
  \end{align*}
  In other words, we made $r$ copies of the layer $\phi(X \cdot)$ with new independent weights, all given the same input $f_{in}$. The downstream layers $f_{out}$ only accept $\theta$ numbers as input, so we introduce an extra linear layer $h \to \bar{h} \in \real^\theta$ with new weights $Z$ to reduce the dimension. We have multiple options for the layer output $\bar{h}$:
  \begin{enumerate}
    \item With weights $Z \in \real^r$, and $h^l_k = \phi(x^l_k \cdot f_{in})$, we can reduce by
      \[
	\bar{h}^l = \sum_{k=1}^r z_k h^l_k,
      \]
      which is equivalent to the simple relaxation \eqref{eq:simple_relaxation}.
    \item With weights $Z \in \real^\theta \otimes \real^r$, we can reduce by 
      \[
	\bar{h}^l = \sum_{k=1}^r z^l_k h^l_k,
      \]
      which is equivalent to the block relaxation \eqref{eq:block-relaxation}.

    \item With $Z \in \real^{\theta \times \theta r}$, we can use a fully connected layer
      \[
        \bar{h} = Z h,
      \]
      which allows even more freedom than the block relaxation \eqref{eq:block-relaxation}.
  \end{enumerate}
  Any of these strategies effectively widen and deepen the original network and are therefore loosely related or over-parametrization.
\end{example}

Although neural networks provide the original motivation for the relaxation idea, we analyze these methods more rigorously for compressed sensing. This area provides non-convex optimization problems as well, but the theoretical background is much better understood. Contrary to \eqref{eq:main-question}, we consider the simplified problem 
\begin{align}
  G(x_1, \dots, x_r, \tilde{z}) & \stackrel{?}{\lesseqqgtr} g(x_j), &
  j & = 1, \dots, r,
  \label{eq:main-question-fixed-x}
\end{align}
where we only optimize the selector $z$. The second and third observation after \eqref{eq:main-question} still apply. In particular, this optimization problem never evaluates all possible $r^\theta$ combinations of $f([x^1_{k_1}, \dots, x^\theta_{k_\theta}])$ but instead is a convex problem in the $r\theta$ dimensional variable $z$. Nonetheless, in Section \ref{sec:cs-relaxation} we show that the relaxed problem can find optimal combinations. One may try to incorporate an $x_j$ optimization as well by a perturbation argument, but this is left for future research.

\subsection{Comparison of Probabilities}
\label{sec:relaxation-probabilities}

In this section, we compare the probabilities to find global optimizers either with $r$ random initial values in \eqref{eq:multiple-initial} or with the full block relaxed optimization problem \eqref{eq:block-relaxation}. The purpose of this discussion is to better understand the prospects of the latter method and therefore, we only consider some informal estimates in a highly idealized scenario. We consider a more rigorous analysis for the compressed sensing in Section \ref{sec:cs-relaxation} below, but for other areas, such as neural networks, it remains unknown to what extend the given estimates are legitimate.

For $r$ random initial values in \eqref{eq:multiple-initial} some natural assumptions are
\begin{enumerate}
  \item There is an ``attractor'' $A$ of the global minimum, meaning that for each initial value $x \in A$ our optimization method of choice (e.g. gradient descent) converges to the global optimum $\min_x \ell[f(x)]$.
  \item Each initial guess $x_k$ is sampled from i.i.d random variables $X_k$.
\end{enumerate}
For the block relaxed optimization problem \eqref{eq:block-relaxation} we assume:
\begin{enumerate}
  \item There are sets $B^l$, $l=1, \dots, \theta$ such that for each initial choice $x^l \in B^l$ and every initial selectors $z^l_k$, $\ell=1, \dots, \theta$, $k=1, \dots, r$, the optimization method of choice (e.g. gradient descent) applied to the block relaxed problem \eqref{eq:block-relaxation} converges to the global optimum $\min_x \ell[f(x)]$ with probability $p_{select}$.
  \item Each initial guess $x^l_k$, $\ell=1, \dots, \theta$, $k=1, \dots, r$ is sampled from i.i.d random variables $X^l_k$.
\end{enumerate}
The first assumption is quite severe and entails that for any initial selectors $z_k^l$ the optimizer can find an optimal selection of the blocks $x^1_{k^1}, \dots, x^\theta_{k^\theta}$ among all possible combinations.  This will be analyzed carefully in the compressed sensing example in Section \ref{sec:cs-relaxation}. For neural networks the assumption is unrealistic because the relaxed network likely has a smaller global minimum than the un-relaxed one. Without changing the arguments below, one can e.g. assume that the optimization of the block relaxed problem converges to a minimum that is smaller than $\min_x \ell[f(x)]$. In order to account for the fact that we may not find an optimal balance of the pieces $f^l(x^l_k)$, $k=1, \dots, r$, we added the extra probability $p_{select}$ to do so successfully.

In the following, we use the abbreviations
\begin{align*}
  p & := P(X_1 \in A), &
  p_l & := P(X^l_1 \in B^l).
\end{align*}
Since all guesses $X_k$ are i.i.d., for the optimization of $r$ repeated trials the probability of success is 
\begin{equation}
  \begin{aligned}
    P(\text{success $r$ trials})
    & = P(\exists k \in \{1, \dots, r\}: \, X_k \in A) \\
    & = 1 - P(\forall k \in \{1, \dots, r\}: \, X_k \not\in A) \\
    & = 1 - \prod_{k=1}^r P(X_k \not\in A) \\
    & = 1 - P(X_1 \not\in A)^r \\
    & = 1 - (1-p)^r.
  \end{aligned}
  \label{eq:successful-guess-probability}
\end{equation}
With the events $SELECT$ that the block relaxation \eqref{eq:block-relaxation} finds the global optimum and $INITIAL := \forall l \in \{1, \dots, \theta\}: \, \exists k \in \{1, \dots, r\}: \, X_k^l \in B^l$ of guessing good initial values, the probability that the block-relaxed optimization \eqref{eq:block-relaxation} is successful is
\begin{align*}
  P(\text{success block relaxation})
  & = P(SELECT \cap INITIAL) \\
  & = P(SELECT | INITIAL) P(INITIAL).
\end{align*}
The first probability of the right hand side is $p_{select}$. With the independence of all blocks $l$, the second can be calculated analogously to \eqref{eq:successful-guess-probability}, which yields
\begin{equation*}
  P(INITIAL) = \prod_{l=1}^\theta P(\exists k \in \{1, \dots, r\}: \, X^l_k \in B^l)
    = \prod_{l=1}^\theta 1 - (1-p_l)^r.
\end{equation*}
and thus
\[
  P(\text{success block relaxation})
  = p_{select} \prod_{l=1}^\theta 1 - (1-p_l)^r.
\]
For easier comparison, let us approximate the success probabilities by some simpler statements.  By a first order Taylor expansion for small $q$ we have $1-q \approx e^{-q}$ and $1-e^{-qr} \approx qr$ and thus
\begin{equation}
  1 - (1-q)^r \approx 1 - (e^{-q})^r = 1-e^{-qr} \approx qr.
  \label{eq:repeat-probability}
\end{equation}
Applied to the success probabilities and assuming that $p_l$ is independent of $l$, we obtain
\begin{align*}
  P(\text{success $r$ trials}) & \approx pr \\
  P(\text{success block relaxation}) & \approx p_{select} (p_l r)^\theta.
\end{align*}
For the sake of comparing the two methods, we assume that $p \approx p_l^\theta$, which can be justified e.g. if $A \approx B^1 \times \cdots \times B^\theta$. Then we have
\begin{equation}
  \begin{aligned}
    P(\text{success $r$ trials}) & \approx p_l^\theta r \\
    P(\text{success block relaxation}) & \approx p_l^\theta (p_{select} r^\theta).
  \end{aligned}
  \label{eq:success-probabilities}
\end{equation}
In conclusion, for $r$ repeated trials we may achieve an $r$ fold increased chance of success and using $r$ block relaxations, which amounts to the same number of total guessed variables, we can hope for a improvement by a factor of $p_{select} r^\theta$. For $p_{select}$ close to one and large $\theta$ this success probability can be significantly larger. However, for the latter result we made quite significant assumptions, which we will only discuss for compressed sensing. In other cases it remains open how much of this potential improvement is realistic.

The above Taylor approximation is a rather crude argument, but in some limiting scenarios the approximations become exact. In order to define the limits properly, first note that the quantities $p$, $r$ and $p_l$ typically depend on some problem parameters such as the dimension of $x$. We denote this parameter by $\gamma$, so that $p = p(\gamma)$ and $r=r(\gamma)$ and $p_l = p_l(\gamma)$. 

We now assume that the success probabilities $p_l$ (or $p)$ go to zero faster than the number of guesses $r$ goes to infinity, i.e. 
\begin{align*}
  \lim_{\gamma \to \infty} p_l & = 0, &
  \lim_{\gamma \to \infty} r & = \infty, &
  \lim_{\gamma \to \infty} p_l r & = 0.
\end{align*}
Then, by Lemma \ref{lemma:prob-compare} (with $q=1/r$) in the Appendix \ref{appendix:limits}, the Taylor approximation \eqref{eq:repeat-probability} used in the derivation of the success probabilities \eqref{eq:success-probabilities} is accurate in the limit

\[
  \lim_{\gamma \to \infty} \frac{1- [1-p_l]^r}{p_lr} = 1
\]
and likewise for $p_l$ replaced by $p$.

\section{Application to Compressed Sensing}
\label{sec:cs-relaxation}

In this section, we consider the block relaxation \eqref{eq:block-relaxation} applied to the compressed sensing problem of Example \ref{example:lp} in some more detail. In Section \ref{sec:cs-block-relaxation} we describe the method and the main result of this paper. Since the result is quite technical, Section \ref{sec:cs-example} provides some more concrete scenarios and connections to the success probabilities in Section \ref{sec:relaxation-probabilities}. Finally, Section \ref{sec:cs-block-relaxation-proof} contains the proof of the main result.

\subsection{Model Problem}
\label{sec:cs-block-relaxation}

Let us first recall Example \ref{example:lp}. For $0 < p \le 1$, a matrix $A \in \real^{m \times N}$ and vector $y \in \real^m$ consider the $\ell_p$-minimization
\begin{equation}
  \begin{aligned}
    & \min_{x \in \real^N} \|x\|_p^p, &
    & \text{s.t.}
    & Ax & = y.
  \end{aligned}
  \label{eq:cs-p-2}
\end{equation}
In addition, upon possibly rescaling the right hand side $y$, we assume without loss of generality that $x \in [-1, 1]^N$, which will simplify our analysis below. 

As in Example \ref{example:lp}, we assume that $N=n\theta$ and split the vector $x$ and sensing matrix $A$ into corresponding blocks
\begin{equation}
  \begin{aligned}
    A & = \begin{pmatrix} A^1 & \cdots & A^\theta \end{pmatrix} \in \real^{m \times n \theta} \\
    x & = \begin{pmatrix} x^1 & \cdots & x^\theta \end{pmatrix} \in \real^{n \theta}.
  \end{aligned}
  \label{eq:block-sensing-A-x}
\end{equation}
Using the block structure, the compressed sensing problem becomes
\begin{equation*}
  \begin{aligned}
    & \min_{x \in \real^N} \sum_{l=1}^\theta \|x^l\|_p^p, &
    & \text{s.t.}
    & \sum_{l=1}^\theta A^l x^l & = y.
  \end{aligned}
\end{equation*}
Repeating the derivation of the block relaxed method \eqref{eq:block-relaxation}, we first make $r$ guesses $x^l_k$ for each block and insert a selection $\|x^l\|_p^p = \sum_{k=1}^r \|x^l_k\|_p^p z^l_k$ with selectors $z^l_k$ into the compressed sensing problem 
\begin{equation}
  \begin{aligned}
    & \min_{z^l \in \{e_1, \dots e_r\}, l=1, \dots, \theta} \sum_{l=1}^\theta \left(\sum_{k=1}^r \|x^l_k\|_p^p |z^l_k| \right) &
    & \text{s.t.}
    & \sum_{l=1}^\theta A^l \left( \sum_{k=1}^r x^l_k z^l_k \right) & = y.
  \end{aligned}
  \label{eq:cs-block-select}
\end{equation}
In the objective function we have used $|z^l_k|$ instead of $z^l_k$ itself because this leads to a standard $\ell_1$-minimization problem after relaxation. As long as $z_k^l$ are standard basis vectors this does not change the problem.

In order to obtain a more compact notation, let $X^l \in \real^{n \times r}$ be the matrices with columns $x^l_k$, $k=1, \dots, r$. In addition, we replace the tensor $z \in \real^r \otimes \real^\theta$ with a corresponding block vector in $\real^{n \theta}$ and obtain the block matrix and vector
\begin{equation*}
  \begin{aligned}
    X & = \begin{pmatrix} X^1 & & \\ & \ddots & \\ & & X^\theta \end{pmatrix} \in \real^{n \theta \times r \theta}
    \\
    z & = \begin{pmatrix} z^1 & \cdots & z^\theta \end{pmatrix} \in \real^{r \theta}.
  \end{aligned}
\end{equation*}
Then, relaxing \eqref{eq:cs-block-select} to any $z^l \in \real^r$, and setting $R := r \theta$, we obtain
\begin{equation}
  \begin{aligned}
    & \min_{z \in \real^R} \sum_{j=1}^{N} \sum_{k=1}^R |X_{jk}|^p |z_k|
    & & \text{s.t.}
    & A X z & = y.
  \end{aligned}
  \label{eq:cs-block-relaxation}
\end{equation}
First note that we omit the optimization of $X_{jk}$ and therefore confine ourselves to the simplified question \eqref{eq:main-question-fixed-x} if the block relaxation can find a good selector $z$ given that $X$ already contains parts of the global solution in its columns. As a consequence, in order to obtain non-zero probabilities to find a solution we will assume that the correct solutions are discrete. This problem is still $NP$ hard in the worst case and discussed in Appendix \ref{sec:np-hard}. The remaining $z$ optimization in \eqref{eq:cs-block-relaxation} is a weighted $\ell_1$ optimization problem and hence convex, unlike the $\ell_p$-minimization we started from. 

Let us now consider to what extend the block relaxed problem can recover the optimal discrete selectors $z^l \in \{e_1, \dots, e_r\}$ in \eqref{eq:cs-block-select} before the relaxation and therefore find the optimal combination $[X^1 z^l, \cdots X^\theta z^\theta]$ of the initial guesses in $X$. The main observation is that the un-relaxed combined selectors $z=[z^1, \dots, z^\theta]$ are $\theta$-sparse and therefore can potentially be recovered by a compressed sensing problem of type \eqref{eq:cs-block-relaxation}. For reference, the following lemma summarizes this idea.

\begin{lemma}
  \label{lemma:block-relaxed-recovery}
  Assume that:
  \begin{enumerate}
    \item For the discrete selectors $\bar{z} = [\bar{z}^1, \dots, \bar{z}^\theta]$ with $\bar{z}^l \in \{e_1, \dots, e_r\}$ the block vector $\bar{x} = [\bar{x}^1, \dots, \bar{x}^\theta] = [X^1 \bar{z}^1, \dots, X^\theta \bar{z}^\theta]$ is a global minimizer of the $\ell_p$-minimization \eqref{eq:cs-p-2}. 

    \item The continuous selector $z = [z^1, \dots, z^\theta]$ is the minimizer of the block relaxed problem \eqref{eq:cs-block-relaxation}, with fixed $X^l$.

  \end{enumerate}

  If the $\theta$-sparse vector $\bar{z}$ is the unique minimizer of the block relaxation \eqref{eq:cs-block-relaxation}, then $z = \bar{z}$ and the reconstruction $Xz$ from the block relaxed method is a global minimizer of the $\ell_p$ minimization \eqref{eq:cs-p-2}.

\end{lemma}

The assumption that the $\theta$-sparse vector $\bar{z}$ is the unique minimizer of \eqref{eq:cs-block-relaxation} is a classical nonuniform recovery statement in compressed sensing and requires some conditions on the sensing matrix 
\[
  A X =
  \begin{pmatrix}
    A^1 X^1 & \cdots & A^\theta X^\theta
  \end{pmatrix}
\]
such as variants of a RIP condition suitable for weighted compressed sensing \cite{RauhutWard2016}. If we have only one block $\theta=1$ and $A$ is orthogonal with $m=n$, then $A X$ satisfies an RIP if $X$ does. On the other hand, if we have the maximal number of blocks with $n=1$ and $\theta =N$, each block is a rank one matrix and a RIP is impossible. Between these two extremes, the matrix $X$ induces some extra randomness into the sensing matrix $AX$, which will help us prove sparse recovery results below, see also \cite{KasiviswanathanRudelson2019}.

In summary, we have two conditions for the matrix $X$: It must contain the global optimizer in its blocks and $AX$ must admit sparse recovery. With the given block structure, the first condition entails that the sensing matrix $AX$ has columns $A^l x^l$, $l = 1, \dots, \theta$. For unfavorable $x$, these may render unique sparse recovery impossible. We avoid this problem with high probability by choosing a random correct sparse $x$ and then define a corresponding right hand side $y = Ax$.

\begin{remark}
  Note that although we assume that $z$ can be uniquely recovered by \eqref{eq:cs-block-relaxation}, this does not imply that the solution $x$ of the $\ell_p$ minimization problem \eqref{eq:cs-p-2} is unique. If $z$ is unique but $x$ is not, this merely implies that $X$, only contains one solution in its range, for otherwise the unique recovery would be violated.
\end{remark}

Throughout the article, for a matrix $C \in \real^{a,b}$, and a subset $R \subset \{1, \dots, b\}$, the matrix $C_{\cdot, R}$ consists of the columns of $C$ with indices in $R$. We are now ready to state the main theorem, which provides the probability that the block relaxed convex problem \eqref{eq:cs-block-relaxation} selects the correct blocks or equivalently $p_{select}$ in \eqref{eq:success-probabilities}.

\begin{theorem}
  \label{th:cs-block-relaxation}

  \begin{enumerate}

    \item Let $x = [x^1, \dots, x^\theta] \in [-1,1]^N$ be a vector with i.i.d random entries on a given support $S \subset \{1, \dots, N\}$ with
    \begin{align}
      \E[x_j] & = 0, &
      \E[x_j^2] & = p_x, &
      j & \in S
      \label{eq:random-x}
    \end{align}
    and let $S^l$ be the indices of $S$ in block $l=1, \dots, \theta$.

    \item Let $X^l \in [-1,1]^{n \times r}$ be matrices that contain the vectors $x^l$ in unknown columns $k^l$ and with remaining entries i.i.d. random numbers with 
    \begin{align}
      \E[X^l_{j,k}] & = 0, &
      \E[|X^l_{j,k}|] & = \nu, &
      k & \ne k^l, &
      l & = 1, \dots, \theta.
      \label{eq:random-X}
    \end{align}
    For abbreviation, let $T = \{k^l | \, l=1, \dots, \theta\}$.

    \item Define the constants
    \begin{align}
      F_S(A)^2 & := \min_{l = 1, \dots, \theta} \|A^l_{\cdot, S^l}\|_F^2, & 
      M(A)^2 & := \max_{l=1, \dots, \theta} \|A^l\|^2
      \label{eq:matrix-conditions}
    \end{align}
    and
    \begin{equation}
      \overline{s} := \max \left\{|S^l|: \, l \in \{1, \dots, \theta\}\right\}.
      \label{eq:def-s-bar}
    \end{equation}

  \end{enumerate}
  Then for any $\alpha \ge 0$ and $0 \le \delta \le 1$ related by 
  \begin{equation}
    1-\delta = \frac{\sqrt{|T|} \overline{s}}{\alpha F_S(A) \sqrt{p_x}}
    \label{eq:delta-alpha}
  \end{equation}
  with probability at least 
  \begin{multline*}
    1 - \Bigg[ 
    2 (R-|T|) \exp \left( - \frac{\nu^2 n^2}{n + 2 M(A)^2\alpha^2} \right)
    \\
  + 2 \left( \frac{12}{\delta} \right)^{|T|} \exp\left(- c \frac{F_S(A)^2}{M(A)^2} \min \left\{\frac{p_x^2 \delta^2}{4 K^4}, \frac{p_x \delta}{2 K^2} \right\} \right) \Bigg]
  \end{multline*}
  the solution $z$ of the block relaxed problem \eqref{eq:cs-block-relaxation} satisfies $x = Xz$, for some positive absolute constants $c$ and $K$.
\end{theorem}

The theorem is proven in Section \ref{sec:cs-block-relaxation-proof}. The first two assumptions of the theorem contain the setup from our relaxation method. We fix a support $S$, sample a corresponding sparse vector $x=[x^1, \dots, x^\theta]$ and define a right hand side $y = Ax$. Then, we try to recover $x$ via the block relaxation \eqref{eq:cs-block-relaxation}. To this end, we sample the ``initial guesses'' or blocks $X^l$ and assume that one combination of columns of $X^l$ contain the original $x$. Thus, the distribution of $x^l$ in the theorem is not the one used for generating the correct solution but rather the conditional distribution $P(x^l_j | \, x^l_j = X_{j,k^l} \, x^l_j \ne 0)$ given that $x^l_j$ is non-zero and matches an entry of the matrix $X$. The theorem then states that with high probability the solution $z$ of the block relaxed problem \eqref{eq:cs-block-relaxation} picks the original $x = Xz$.

The conditions on the matrix $A$ enter via the constants \eqref{eq:matrix-conditions} and are comparatively weak. The quantity $F_S(A)^2/M(A)^2$ is related to the stable rank $\|A^l\|_F^2 / \|A^l\|^2$ of a matrix and equals to $\min_l|S^l|$ if all blocks $A^l$ have orthonormal columns, see e.g. \cite{KasiviswanathanRudelson2019} for more information. Note that all constants that depend on $A$ only do so via individual blocks $A^l$ and are independent of any relation between different blocks $A^l$ and $A^j$, $j \ne l$. Therefore, the given conditions are much weaker than a RIP and allow e.g. repetitions $A^l = U$ of identical unitary matrices. We consider those conditions more closely for a more concrete example in Section \ref{sec:cs-example} below.

The theorem is reminiscent of standard sparse recovery results, with weighted $\ell_1$-norm and a special structure in the sensing matrix $AX$. Moreover, it provides a positive answer to the question raised in \eqref{eq:main-question-fixed-x} if the block relaxation can find the optimal combination among the $r^\theta$ possible combinations of blocks $x^l_k$. To this end, first note that the block relaxed optimization problem \eqref{eq:cs-block-relaxation} is convex and can be efficiently solved with e.g. reweighted least squares algorithms. The resulting vector $x = Xz$ is a $|S|$-sparse solution of the constraint equation $Ax = y$. In many applications this is already what we want. Anyways, if in addition the original non-convex $\ell_p$-minimization problem \eqref{eq:cs-p-2} has a unique $|S|$-sparse solution, it must be $x = Xz$ and we also have a solution of the non-convex optimization problem. This was our original question in \eqref{eq:main-question-fixed-x} and requires weaker conditions than the classical $RIP$ for $\ell_1$-minimization, see e.g. \cite {ChartrandStaneva2008,FoucartLai2009,Sun2012}.

Finally, the assumption that the blocks $X^l$ contain a correct guess is quite severe or rather unlikely, even for the discrete cases we consider below. However, recall the purpose of this result is to investigate the feasibility of the relaxation argument from Section \ref{sec:relaxation} in a simple model problem, not to provide a practical algorithm. That would certainly include an optimization of the blocks $X^l$ as in the general block relaxed problem \eqref{eq:block-relaxation}. In addition, the assumption on the sensing matrix $A$ still allow polynomial-time reductions of $NP$-hard problems to $\ell_0$-minimization, although only smaller instances than for general matrices $A$, see Appendix \ref{sec:np-hard} for a discussion.

\subsection{Recovery Probability}
\label{sec:cs-example}

In order to disentangle all requirements and statements in \ref{th:cs-block-relaxation}, in this section, we consider some more specific scenarios. To this end, in the following let $\gtrsim$, $\lesssim$ and $\sim$ denote greater, smaller and equivalence up to some generic constants independent of the problem dimensions, sparsity and expectations such as $\nu$ or $p_x$.

First, we assume that the support $S$ is equidistributed among the blocks, i.e. that there is some $s$ with $|S| = \theta s$ and 
\[
  |S^l| \sim s.
\]
If we choose $S$ uniformly at random, this is satisfied with high probability for sufficiently large $s$. Indeed, $|S^l|$ is distributed by a hypergeometric distribution so that the observation easily follows from Chebyshev's inequality.

Next assume that $m = n$, so that the blocks $A^l$ are square matrices, and that the columns of $A^l$ are (almost) orthonormal. This implies that
\begin{align*}
  \|A_l\| & \sim 1, & 
  \|A^l\|_F^2 & \sim n, & 
  \|A^l_{\cdot, S^l}\|_F^2 & \sim s
\end{align*}
and therefore
\begin{align*}
  F_S(A)^2 & \sim s, &
  \overline{s} & \sim s, &
  M(A) & \sim 1, &
  \frac{F_S(A)^2}{M(A)^2} & \sim s.
\end{align*}
There are no relations between the columns of different blocks, so e.g. it is legitimate to choose all blocks equal, which clearly violates the RIP condition.

Finally, we assume that we have some good a-priory knowledge of the size $|S|$ and choose 
\begin{align*}
  \nu & \sim s/n, &
  p_x & \sim 1.
\end{align*}
The second probability of $p_x$ states that given the information that we are on the support of $x$, the entries are not overly strongly clustered around zero. With these constants, the probability of recovery failure of Theorem \eqref{th:cs-block-relaxation} is at most
\begin{equation}
  2 \exp \left( - c \frac{s^2}{n + M(A)^2\alpha^2} + \ln(R) \right)
  + 2 \exp\left(- c s \min \left\{\frac{p_x^2 \delta^2}{4 K^4}, \frac{p_x \delta}{2 K^2} \right\}  + |T| \ln \left(\frac{12}{\delta} \right) \right) 
\end{equation}
for some generic constant $c$, which may differ from the one in the theorem and change in the calculations below. The algorithm does not depend on the choice of $\alpha$ and $\delta$, which we can now choose to bound this failure probability. To this end, let us choose $\delta \sim (1-\delta) \sim 1$ so that by \eqref{eq:delta-alpha} and $|T|=\theta$, we have
\begin{equation*}
  1 \sim (1-\delta) \sim \frac{\sqrt{|T|} \overline{s}}{\alpha F_S(A)} \sim \frac{\sqrt{\theta s}}{\alpha}
  \quad \Leftrightarrow \quad
  \alpha \sim \sqrt{\theta s}.
\end{equation*}
Thus, the failure probability reduces to 
\begin{equation}
  2 \exp \left( - c \frac{s^2}{n + M(A)^2 \theta s} + \ln(R) \right)
  + 2 \exp\left(- c s + C \theta \right) =: (I) + (II)
  \label{eq:cs-failure-probability}
\end{equation}
for some new generic constant $C$. Using $M(A) \sim 1$, we have $\frac{s^2}{n + M(A)^2 \theta s} \gtrsim \min \left\{ \frac{s^2}{n}, \, \frac{s}{\theta} \right\}$, which must be larger than $\ln(R) = \ln(r\theta)$ for the exponent of $(I)$ to be negative. Therefore, we must have
\begin{equation}
  r \lesssim \min \left\{\frac{1}{\theta} e^{\frac{s}{\theta}}, \frac{1}{\theta} e^{\frac{s^2}{n}} \right\}.
  \label{eq:max-trials}
\end{equation}
The first component of the minimum ensures that $s \gtrsim \theta$, which implies that also $(II)$ has a negative exponent.

First note that the condition \eqref{eq:max-trials} limits the number $r = R/\theta$ of possible trials. Depending on the relative sizes of $s$, $n$ and $\theta$, this number can be exponentially large and the block relaxed scheme is able to correctly select an exponentially large number of pieces $x^l_{k^1}, \dots, x^l_{k_\theta}$ contained in the guesses $X^l$ for a non-linear problem.

Second, the condition \eqref{eq:max-trials} implies that the sparsity $s$ per block cannot be too small. The reason is that we have very limited assumptions on $A$. In particular the sensing matrix $AX$ contains the columns $A^l x^l$. If $x^l$ is overly sparse, this does not guarantee enough randomness to ensure sparse recovery.

With the probabilities $p_l$ that the blocks of $X^l$ contain the solution blocks $x^l$ and the success probability $p_{select}$ of sparse recovery from \eqref{eq:cs-failure-probability}, by the arguments in Section \ref{sec:relaxation-probabilities}, we have the probabilities
\begin{equation*}
  \begin{aligned}
    P(\text{success $r$ trials}) & \approx p_l^\theta r \\
    P(\text{success block relaxation}) & \approx p_l^\theta (p_{select} r^\theta)
  \end{aligned}
\end{equation*}
to recover a $|S|$-sparse vector with constraint $Ax = y$, with very weak conditions on $A$. If this matrix allows unique sparse recovery from $\ell_p$-minimization, it is also the global minimizer of \eqref{eq:cs-p-2}. Given the conditions in \eqref{eq:max-trials}, we can ensure that $p_{select}$ is close to one so that the block relaxation provides a $r^\theta/r$ enhanced chance to find the solution over $r$ repeated guesses.

In order for the probability $p_l$ to be non-zero, we need to sample $x^l$ and $X^l$ from discrete distributions. Even with this restrictive assumption, $p_l$ is still negligible and the resulting success probability of block relaxation is excessively small. This is not fully unexpected because with the given assumptions on the sensing matrix $A$ and discrete $x$ in e.g. $\{-1,-1/2,0,1/2,1\}$ the $\ell_p$-minimization problem is still $NP$-hard in general, see Appendix \ref{sec:np-hard}. Also recall that for a practical algorithm we would incorporate the blocks $X^l$ into the optimization as in \eqref{eq:block-relaxation}, which removes the requirement to correctly guess the blocks $x^l$ in one shot and with it the requirement of $x^l$ to be discrete.

\subsection{Proof of Theorem \ref{th:cs-block-relaxation}}
\label{sec:cs-block-relaxation-proof}

In this section, we prove Theorem \ref{th:cs-block-relaxation}. The proof follows standard lines for sparse recovery results, with some slight twists for the added structure. In Section \ref{sec:proof-setup}, we first introduce some notations and setup used throughout the entire section. Then, we show concentration estimates (Section \ref{sec:concentration}), RIP type results for fixed sparse subsets $S$ (Section \ref{sec:rip}) and then finally combine these results for a non-uniform recovery argument in Section \ref{sec:recovery}.

\subsubsection{Notations and Setup}
\label{sec:proof-setup}

Let all assumptions from Theorem \ref{th:cs-block-relaxation} be satisfied. We calculate the sparse recovery probability given that the blocks $x^l$ are contained in the columns of the respective matrices $X^l$. W.l.o.g, we assume that $x^l$ are always the first columns so that $X$ has the block structure
\begin{align}
  X & = \begin{pmatrix} x^1 &  \bar{X}^1 & & & \\ & \ddots & \ddots & & \\ & & x^\theta & \bar{X}^\theta \end{pmatrix}, &
  x^l & \in \real^{n \times 1}, & 
  \bar{X}^l & \in \real^{n \times r-1},
  \label{eq:sample-given-sparsity}
\end{align}
where $\bar{X}^l$ are i.i.d. random matrices. 

\begin{remark}
  In Theorem \ref{th:cs-block-relaxation}, we show a sparse recovery result only with high probability. Therefore, we must ensure that the matrices $X$ with a given sparsity pattern $S$ in one column $x^l$ are not included in the low probability set where spare recovery fails. Hence, we make this patter explicit in our proof.
\end{remark}

By assumptions \eqref{eq:random-x} and \eqref{eq:random-X} of Theorem \ref{th:cs-block-relaxation}, we have the following expectation, variances and $\psi_2$-norms:
\begin{equation}
  \begin{aligned}
    \E[x_j] & = 0, &
    \E[x_j^2] & = p_x, &
    \|x_j\|_{\psi_2} & \le K, & 
    j & \in S
    \\
    \E[X^l_{jk}] & = 0, &
    \E[(X^l_{jk})^2] & = p_X,&
    \|X^l_{jk} \|_{\psi_2} & \le K, &
    & \begin{array}{ll}
      j & = 1, \dots, n, \\
      k & = 1, \dots r-1
    \end{array},
  \end{aligned}
  \label{eq:expectations-assumption}
\end{equation}
for some $p_X \ge 0$ constant $K>0$ and $\psi_2$-norm defined by $\|x\|_{\psi_2} := \sup_{a \ge1} a^{-1/2} (\E[|x|^a])^{1/a}$, see e.g. \cite{RudelsonVershynin2013}. Note that all variances and $\psi_2$ norms are bounded because by assumption the entries of $X^l$, including the first column $x^l$, are in the interval $[-1,1]$.

\subsubsection{Concentration Estimates}
\label{sec:concentration}

In this section, we state concentration estimates for $\|AXu\|$ for some $u\in \real^R$. To this end, let us split an arbitrary vector $u \in \real^R$ according to the block structure \eqref{eq:sample-given-sparsity} of $X$ as 
\[
  u := \begin{pmatrix}
    v^1 & u^1 & \cdots & v^\theta & u^\theta
  \end{pmatrix}^T,
\]
with $v^l \in \real$ and $u^l \in \real^{r-1}$. Then the concentration inequality is shown with respect to the weighted norm
\begin{equation}
  \|u\|_A^2 :=  \|W_A u\| := \sum_{l=1}^\theta \left\{ p_x \|A_{\cdot, S^l}\|_F^2 |v^l|^2 + p_X \|A^l\|_F^2 \|u^l\|^2 \right\}.
  \label{eq:A-norm}
\end{equation}
with diagonal weight matrix
\begin{equation}
  W_A = \operatorname{diag}(\sqrt{p_x} \|A_{\cdot, S^1}\|_F, \sqrt{p_X} \|A^1\|_F, \dots, \sqrt{p_x} \|A_{\cdot, S^\theta}\|_F, \sqrt{p_X} \|A^\theta\|_F).
  \label{eq:A-norm-weight}
\end{equation}
Recall that $S^l$ are the indices in $S$ contained in the block $X^l$ of $X$. For the remainder of this section $c$ denotes a positive absolute constant.

\begin{proposition}
  \label{prop:factorization-concentration}
  Let $A$ and $X$ be the matrices defined in \eqref{eq:block-sensing-A-x} and \eqref{eq:sample-given-sparsity} with independent sub-Gaussian entries satisfying \eqref{eq:expectations-assumption}. Then, for every $u \in \real^R$ and $\epsilon \ge 0$,
  \begin{equation*}
    \operatorname{Pr} \left[ \left| \|AXu\|^2 - \|u\|_A^2 \right| \ge \epsilon F^2
     \right] 
     \le 2 \exp \left( -c \frac{F_S(A)^2}{M(A)^2} \min \left\{ \frac{\epsilon^2}{K^4}, \frac{\epsilon}{K^2} \right\} \right)
  \end{equation*}
  with
  \[
    F^2 = \sum_{l=1}^\theta \left\{ \|A_{\cdot, S^l}\|_F^2 \|v^l\|^2 + \|A^l\|_F^2 \|u^l\|^2 \right\}
  \]
  and $F_S(A)$ and $M(A)$ defined in \eqref{eq:matrix-conditions}.

\end{proposition}

We only need a corollary of this proposition for $u$ restricted to the support set $T$ of the selectors $z$.

\begin{corollary}
  \label{cor:factorization-concentration}
  Let $A$ and $X$ be the matrices defined in \eqref{eq:block-sensing-A-x} and \eqref{eq:sample-given-sparsity} with independent sub-Gaussian entries satisfying \eqref{eq:expectations-assumption}. Then, for every $u \in \real^R$ supported on $T$ and $\epsilon \ge 0$
  \begin{equation*}
    \operatorname{Pr} \left[ \left| \|AXu\|^2 - \|u\|_A^2 \right| \ge \epsilon \|u\|_A^2 \right] \le 2 \exp \left( -c \frac{F_S(A)^2}{M(A)^2} \min \left\{ \frac{p_x^2 \epsilon^2}{K^4}, \frac{p_x \epsilon}{K^2} \right\} \right),
  \end{equation*}
  with $F_S(A)$ and $M(A)$ defined in \eqref{eq:matrix-conditions} and $p_x$ defined in \eqref{eq:random-x}.
\end{corollary}

\begin{proof}
  Since $u$ is supported on $T$, we have $\|u^l\|_2^2 = 0$ for all $l=1, \dots, \theta$ and therefore the definition of $F$ in Proposition \ref{prop:factorization-concentration} and the definition \eqref{eq:A-norm} of the $\|\cdot\|_A$-norm yield
\begin{align*}
  F^2
  & = \sum_{i=1}^\theta \|A_{\cdot, S^l}\|_F^2 \|v^l\|^2, &
  \|u\|_A^2,
  & = \sum_{i=1}^\theta p_x \|A_{\cdot, S^l}\|_F^2 \|v^l\|^2.
\end{align*}
Thus, we have $p_x F^2 = \|u\|_A^2$, which proves the corollary.

\end{proof}

The proof of Proposition \ref{prop:factorization-concentration} is similar to \cite{KasiviswanathanRudelson2019}, and uses the following corollary of the Hanson-Wright inequality, see Appendix \ref{appendix:hanson-wright} for more details.

\begin{corollary}
  \label{cor:hanson-wright}
  Let $v \in \real^d$ be a vector with independent components with $\E[v_i] = 0$ and $\|v_i\|_{\psi_2} \le K$ and $C^T C \in \real^{d \times d}$ be a matrix. Then, for every $t \ge 0$,
  \begin{equation*}
    \operatorname{Pr} \left[ \left| v^T C^T C v - \E[v^T C^T C v] \right| \ge \epsilon \|C\|_F^2 \right] 
    \le 2 \exp \left( -c \frac{\|C\|_F^2}{\|C\|^2} \min \left\{ \frac{\epsilon^2}{K^4}, \frac{\epsilon}{K^2} \right\} \right).
  \end{equation*}
\end{corollary}

In order to apply the corollary, we construct a vectorization $\hat{X}$ of the matrix $X$ and a matrix $B$ with $\hat{X}^T B^T B \hat{X} = \|A X u\|^2$. Let us first consider this vectorization for a generic matrix $M \in \real^{a \times b}$, vector $w \in \real^{c}$ and random matrix $R \in \real^{b \times c}$ with i.i.d entries, expectation $\E[r_{ij}] = 0$ and variance $\E[r_{ij}^2] = V$. By Appendix \ref{appendix:vectorization}, we identify $R$ with the tensor $\hat{R} \in \real^b \otimes \real^c$ and have
\begin{equation}
   M R w = (M \otimes w^T) \hat{R}
   \label{eq:vectorization}
\end{equation}
and 
\begin{equation}
  \E \left[ \|M R w\|^2 \right] = V \|M\|_F^2 \|w\|^2.
  \label{eq:expectation-matrix-matrix-vector-prod}
\end{equation}

Let us now construct the vectorization $\hat{X}$ and $B$ with $B \hat{X} = AXu$. Instead of applying the last two identities directly, we are a little more careful with regard to the block structure. $\hat{X}$ is defined by
\begin{equation*}
  \hat{X} 
  := \begin{pmatrix} \hat{x}^1 & \hat{X}^1 & \cdots & \hat{x}^\theta & \hat{X}^\theta \end{pmatrix} 
  \in \bigtimes_{l=1}^\theta \left( \real^{|S^l|} \times \real^{n(r-1)} \right)  
\end{equation*}
where the $\hat{x}^l$ and $\hat{X}^l := \hat{\bar{X}}^l$ are the vectorizations of the restriction $x^l_{S^l}$ of $x^l$ to its support $S^l$, and $X^l$, respectively. Likewise, $B$ is defined by
  \begin{align}
  B & 
  := \begin{pmatrix} 
    A_{\cdot, S^1} \otimes (v^1)^T 
    & A^1 \otimes (u^1)^T 
    & \cdots 
    & A_{\cdot, S^\theta} \otimes (v^\theta)^T 
    & A^\theta \otimes (u^\theta)^T \end{pmatrix}.
  \label{eq:matrix-vectorized}
\end{align}
where $v^l$ is considered as a $1 \times 1$ matrix. Then, by \eqref{eq:vectorization} the vectorization of the product $AXu$ is given by
\begin{equation}
  B \hat{X} 
  = \sum_{l=1}^\theta (A_{\cdot, S^l} \otimes (v^l)^T) \hat{x}^l + (A^l \otimes (u^l)^T) \hat{X}^l
  = \sum_{l=1}^\theta A_{\cdot, S^l} x_{S^l}^l v^l + A^l X^l u^l
  = AXu.
  \label{eq:factorization-vectorization}
\end{equation}
This allows us to prove Proposition \ref{prop:factorization-concentration} with Corollary \ref{cor:hanson-wright} of the Hanson-Wright inequality.

\begin{proof}[Proof of Proposition \ref{prop:factorization-concentration}]

From the vectorization \eqref{eq:factorization-vectorization} we have
\[
  \|AXu\|^2 = \|B \hat{X}\|^2 = \hat{X}^T B^T B \hat{X}.
\]
so that we can use Corollary \ref{cor:hanson-wright} of the Hanson-Wright inequality to show concentration inequalities for $\|AXu\|^2$. To this end, in the following we compute all terms in the Corollary. We start with the expectation value:
\begin{align*}
  \E \left[ \|AXu\|^2 \right]
  & = \E \left[ \left\|\sum_{l=1}^\theta \left\{A^l x^l v^l + A^l \bar{X}^l u^l\right\} \right\|^2 \right]
  \\
  & = \sum_{l=1}^\theta \left\{\E \left[\left\|A^l x^l v^l \right\|^2 \right] + \E \left[ \left\| A^l \bar{X}^l u^l\right\|^2 \right] \right\},
\end{align*}
where we have used that because of independence and zero mean of the entries, all cross terms $\dualp{A^l \bar{X}^l u^l, A^j \bar{X}^j u^j}$ and $\dualp{A^l x^l v^l, A^j x^j v^j}$ for $l \ne j$ and $\dualp{A^l \bar{X}^l u^l, A^j x^j v^j}$ for all $l,j$ vanish. 

Using the zero-mean property and the variances defined in \eqref{eq:expectations-assumption} and applying \eqref{eq:expectation-matrix-matrix-vector-prod} yields
\begin{align*}
  \E \left[\left\|A^l x^l v^l \right\|^2 \right] = \E \left[\left\|A_{\cdot, S^l} x^l_{S^l} v^l \right\|^2 \right] & = p_x \|A_{\cdot, S^l}\|_F^2 |v^l|^2 
  \\
  \E \left[ \|A^l \bar{X}^l u^l\|^2 \right] & = p_X \|A^l\|_F^2 \|u^l\|^2.
\end{align*}
In conclusion, we have
\begin{equation}
  \E \left[ \|AXu\|^2 \right] = \sum_{i=1}^m \left\{ p_x \|A_{\cdot, S^l}\|_F^2 |v^l|^2 + p_X \|A^l\|_F^2 \|u^l\|^2 \right\} = \|u\|_A^2.
  \label{eq:hanson-wright-expectation}
\end{equation}
Note that if we could normalize both $\|A_{\cdot, S^l}\|_F$ and $\|A^l\|_F$ to one, the right hand side would reduce to $\|u\|^2$. However that is not possible because $A_{\cdot, S^l}$ is a sub-matrix of $A^l$. 

The next quantity in Corollary \ref{cor:hanson-wright} is the Frobenius norm 
\begin{equation}
  \|B\|_F^2 
  = \sum_{l=1}^\theta \left\{ \|A_{\cdot, S^l}\|_F^2 \|v^l\|^2 + \|A^l\|_F^2 \|u^l\|^2 \right\}
  = F^2.
  \label{eq:hanson-wright-frobenius}
\end{equation}
The spectral norm can easily be computed with \eqref{eq:block-matrix-norms} in the appendix, which yields
\[
  \|B\|^2 \le \sum_{l=1}^\theta \|A_{\cdot, S^l}\|^2 \|v^l\|^2 + \|A^l\|^2 \|u^l\|^2.
\]
Together with \eqref{eq:hanson-wright-frobenius} and using that $\|A_{\cdot, S^l}\|_F \le \|A^l\|_F$ and $\|A_{\cdot, S^l}\| \le \|A^l\|$ this yields
\begin{equation}
  \frac{\|B\|_F^2}{\|B\|^2}
  \ge \frac{\min_{l=1,\dots,\theta} \|A_{\cdot, S^l}\|_F^2}{\max_{l=1,\dots,\theta} \|A^l\|^2} = \frac{F_S(A)^2}{M(A)^2}.
  \label{eq:hanson-wright-spectral-rank}
\end{equation}
We have calculated all terms in Corollary \ref{cor:hanson-wright} of the Hanson-Wright inequality, which implies
\begin{equation*}
  \operatorname{Pr} \left[ \left| \|AXu\|^2 - \|u\|_A^2 \right| \ge \epsilon \|B\|_F^2 \right] \le 2 \exp \left( -c \frac{\|B\|_F^2}{\|B\|^2} \min \left\{ \frac{\epsilon^2}{K^4}, \frac{\epsilon}{K^2} \right\} \right),
\end{equation*}
which by \eqref{eq:hanson-wright-frobenius} and \eqref{eq:hanson-wright-spectral-rank} proves the proposition.

\end{proof}

\subsubsection{RIP Type Estimates}
\label{sec:rip} 

We show a RIP like estimate, only for one fixed sparse set $T \subset\{1, \dots, R\}$. The result and proof are identical to \cite[Lemma 5.1]{BaraniukDavenportDeVoreEtAl2008} only with the $\ell_2$-norm replaced by the $\|\cdot\|_A$-norm.

\begin{lemma}
  \label{lemma:RIP-fixed-support}
  Let all assumptions of Corollary \ref{cor:factorization-concentration} be satisfied. Then for the set $T \subset \{1, \dots, R\}$ containing the columns $x^l$ and $0 < \delta < 1$, we have
  \[
    (1-\delta) \|u\|_A \le \|AXu \| \le (1+\delta) \|u\|_A
  \]
  for all $u \in \real^R$ supported on $T$ with probability at least
  \begin{equation}
    1 - 2 \left( \frac{12}{\delta} \right)^{|T|} \exp\left(- c \frac{F_S(A)^2}{M(A)^2} \min \left\{\frac{p_x^2 \delta^2}{4 K^4}, \frac{p_x \delta}{2 K^2} \right\} \right).
    \label{eq:RIP-fixed-support-probability}
  \end{equation}
\end{lemma}

\begin{proof}
 
  Let $U_T \subset \real^R$ be the vectors with support contained in $T$. Then, there is a $\delta/4$ cover $Q_T$ of the unit sphere in $U_T$ with respect to the $\|\cdot\|_A$-norm with $|Q_T| \le (12/\delta)^{|T|}$, see e.g. \cite{LorentzGolitschekMakovoz1996,FoucartRauhut2013}. From the concentration inequality Corollary \ref{cor:factorization-concentration} with $\epsilon= \delta/2$, together with a union bound, we have that
\[
  \left( 1-\frac{\delta}{2} \right) \|u\|_A^2 \le \left\|AXu \right\|^2 \le \left( 1+\frac{\delta}{2} \right) \|u\|_A^2
\]
with probability at least \eqref{eq:RIP-fixed-support-probability}. This is analogous to \cite[(5.4)]{BaraniukDavenportDeVoreEtAl2008}. Using the remainder of the proof in the reference verbatim, shows that
\[
  (1-\delta) \|u\|_A \le \left\|AXu \right\| \le (1+\delta) \|u\|_A
\]
for all $u$ supported on $T$, which completes the proof.

\end{proof}

\begin{corollary}
  \label{cor:RIP-fixed-support}
  Let all assumptions of Lemma \ref{lemma:RIP-fixed-support} be satisfied and assume that the weight matrix $W_A$ of the $\|\cdot\|_A$-norm defined in \eqref{eq:A-norm-weight} is invertible. Then, with probability at least \eqref{eq:RIP-fixed-support-probability} the singular values $\sigma_i$ of the matrix $(AXW_A^{-1})_{\cdot, T}$ satisfy
  \[
    1-\delta \le \sigma_i \le 1+\delta.
  \]
  
\end{corollary}

\begin{proof}

With the definition \eqref{eq:A-norm-weight} of $W_A$, Lemma \ref{lemma:RIP-fixed-support} implies that
\[
  (1-\delta) \|W_A u\| \le \|AXu\| \le (1+\delta) \|W_A u\|
\]
with the given probability \eqref{eq:RIP-fixed-support-probability} for all $u$ with support $T$. With $z := W_A u$, this implies
\[
  (1-\delta) \|z\| \le \|AX W_A^{-1}z\| \le (1+\delta) \|z\|.
\]
Choosing right singular vectors of $AX W_A^{-1}$ restricted to columns in $T$ for $z$, directly yields the result.
  
\end{proof}

\subsubsection{Sparse Recovery}
\label{sec:recovery} 

The remaining proof of the sparse recovery Theorem \ref{th:cs-block-relaxation}, is analogous to non-uniform sparse recovery as in e.g. \cite[Theorem 9.16]{FoucartRauhut2013}.

By the assumptions of Theorem \ref{th:cs-block-relaxation}, the vectors $x^l$ are contained as columns in the blocks $X^l$. We denote the indices of these columns as $T \subset{1, \dots, R}$. In \eqref{eq:sample-given-sparsity} above, we have w.l.o.g. assumed that these are the first columns in the respective blocks $X^l$. Note, however, that this choice was only for notational convenience and in general $T$ is unknown, except for some rudimentary properties like $t := |T| = \theta$. In addition, note that the set $T$ coincides with the support of the selector $z$ and the mayor goal of the sparse recovery problem \eqref{eq:cs-block-relaxation} is to find this vector.

In the following, let $W_\ell \in \real^{R \times R}$ be the diagonal matrix with $(W_\ell)_{kk} = \|X_{\cdot, k}\|_p^p$, which constitutes the weights in the weighted $\ell_1$-minimization \eqref{eq:cs-block-relaxation} and $W_{\ell T}$ the restriction to the index set $T$. On this special index set, we have
\begin{equation}
  (W_\ell)_{kk} = \|X_{\cdot, k}\|_p^p = \|x^l\|_p^p \le |S^l|,
  \label{eq:weight-on-support}
\end{equation}
if $k$ is in the block $l$, where we have used that $x^l$ has entries in the interval $[-1, 1]$ on its support.

In order to simplify the notations, for any matrix $C$, let $C^{+*} = (C^*)^+ = (C^+)^*$ be the adjoint of the pseudo inverse.

Before we prove the main result Theorem \ref{th:cs-block-relaxation}, we need two more lemmas.

\begin{lemma}
  \label{lemma:layer-recovery-vector-bound}
  For any $\alpha \ge 0$ and $0 \le \delta \le 1$ satisfying \eqref{eq:delta-alpha} and for any $u \in \real^R$ with support on $T$, we have
  \begin{multline}
    P\left(\alpha \le \|(AX_T)^{+*} W_{\ell T} \sign(u_T)\| \right)
    \\
    \le 2 \left( \frac{12}{\delta} \right)^{|T|} \exp\left(- c \frac{F_S(A)^2}{M(A)^2} \min \left\{\frac{p_x^2 \delta^2}{4 K^4}, \frac{p_x \delta}{2 K^2} \right\} \right)
    \label{eq:pseudo-inverse-estimate}
  \end{multline}
  for constants $c,K$ from Corollary \ref{cor:RIP-fixed-support}.

\end{lemma}

\begin{proof}

  Let us use the abbreviations
  \begin{align*}
    v & := (AX_T)^{+*} W_{\ell T} \sign(u_T), &
    W_{AT} & := (W_A)_{\cdot, T}, &
    F & := F_S(A)
  \end{align*}
  for the weight matrix $W_A$ of the $\|\cdot\|_A$ norm defined in \eqref{eq:A-norm-weight}. Then, the left hand side of \eqref{eq:pseudo-inverse-estimate} becomes $P(\alpha \le \|v\|)$. Before we estimate this probability, we calculate an estimate for $\|v\|$. By the definition \eqref{eq:matrix-conditions} of $F = F_S(A)$ and the definition of $T$ we have $\|W_{AT}^{-1}\| \le 1/(F\sqrt{p_x})$. Let $\sigma_{min}$ be the smallest singular value of $AX_T W_{AT}^{-1}$. Since $W_{AT}$ is invertible, we have
\[
  v \in \operatorname{range}[(AX_T)^{+*}] = \operatorname{ker}[(AX_T)^*]^\perp = \operatorname{ker}[W_{AT}^{-1} (AX_T)^*]^\perp = \operatorname{ker}[(AX_T W_{AT}^{-1})^*]^\perp
\]
and therefore
\[
  \|v\| 
  \le \frac{1}{\sigma_{min}} \|(AX_T W_{AT}^{-1})^* v\|
  = \frac{1}{\sigma_{min}} \|W_{AT}^{-1}(AX_T)^* v\|
  \le \frac{1}{\sigma_{min} F \sqrt{p_x}} \|(AX_T)^* v\|.
\]
Plugging in the definition of $v$ and using that $(AX_T)^* (AX_T)^{+*}$ is an orthogonal projector with matrix norm bounded by one, we conclude that
\[
  \|v\| 
  \le \frac{1}{\sigma_{min} F \sqrt{p_x}} \|W_{\ell T} \sign(u_T)\|
  \le \frac{\overline{s}}{\sigma_{min} F \sqrt{p_x}} \|\sign(u_T)\|
  = \frac{\sqrt{|T|} \overline{s}}{\sigma_{min} F \sqrt{p_x}},
\]
where in the second inequality we have used \eqref{eq:weight-on-support} and the definition \eqref{eq:def-s-bar} of $\overline{s}$.

We now proceed with the estimate of the probability in the left hand side of \eqref{eq:pseudo-inverse-estimate}. For any $\alpha \ge 0$, we have
\[
  \alpha \le \|v\| 
  \le \frac{\sqrt{|T|} \overline{s}}{\sigma_{min} F \sqrt{p_x}}
\]
and thus using the assumption \eqref{eq:delta-alpha} in the last identity
\begin{multline*}
  P\left(\alpha \le \|(AX_T)^{+*} W_{\ell T} \sign(u_T)\| \right)
  = P(\alpha \le \|v\|)
  \\
  \le P\left( \alpha \le \frac{\sqrt{|T|} \overline{s}}{\sigma_{min} F \sqrt{p_x}} \right)
  = P\left( \sigma_{min} \le \frac{\sqrt{|T|} \overline{s}}{\alpha F \sqrt{p_x}} \right)
  = P(\sigma_{min} \le 1-\delta).
\end{multline*}
The latter probability is smaller, than the probability that there is any singular value that is not contained in the interval $[1-\delta, 1+\delta]$ and thus Corollary \ref{cor:RIP-fixed-support} implies \eqref{eq:pseudo-inverse-estimate}.

\end{proof}

\begin{lemma}
  \label{lemma:inner-product-size}
  Let $x \in [-1, 1]^d$, $d \ge 1$ be a random vector with zero mean and expectation $\E[|x_i|] = \nu$, $i=1, \dots, d$. Then for any $v \in \real^d$, we have
  \[
    P \left( \dualp{x, v} \ge \|x\|_p^p \right)
    \le 2 \exp \left( - \frac{\nu^2 d^2}{d + 2 \|v\|_2^2} \right).
  \]
\end{lemma}

\begin{proof}
  
Since $-1 \le x_i \le 1$ and $p \le 1$, we have $|x_i|^p \ge |x_i|$ so that $\|x\|_p^p \ge \|x\|_1$ and therefore
  \begin{multline*}
    P \left( |\dualp{x, v}| \ge \|x\|_p^p \right)
    \le P \left( |\dualp{x, v}| \ge \|x\|_1 \right)
    \\
    \le P \left( \dualp{x, v} \ge \|x\|_1 \right) + P \left( \dualp{-x, v} \ge \|x\|_1 \right).
  \end{multline*}
  It suffices to estimate the first summand in the right hand side, the other follows analogously. We have
  \[
    P \left( \dualp{x, v} \ge \|x\|_1 \right)
    = P \left( \sum_{i=1}^d [x_i v_i - |x_i|  + \nu] \ge \nu d \right)
    =: P \left( \sum_{i=1}^d X_i \ge \nu d \right),
  \]
  with $X_i := x_i v_i - |x_i| + \nu$. By construction, $X_i$ has zero mean and from $X_i - \nu = |x_i| (\operatorname{sign}(x_i) v_i - 1)$ and $-1 \le x_i \le 1$, we obtain
  \[
    - |v_i| - 1 \le X_i - \nu  \le \max\{0, |v_i|-1\}
  \]
  so that $X_i$ is contained in an interval of length
  \[
    w_i = \max\{1+|v_i|, 2|v_i|\}.
  \]
  It follows that $w_i^2 \le 2 + 4v_i^2$ and therefore, Hoeffding's inequality implies
  \[
    P \left( \dualp{x, v} \ge \|x\|_1 \right)
    \le \exp \left( - \frac{\nu^2 d^2}{d + 2 \|v\|_2^2} \right).
  \]
  Using the same estimate for $P \left( \dualp{-x, v} \ge \|x\|_1 \right)$, concludes the proof.

\end{proof}

We are now ready to prove Theorem \ref{th:cs-block-relaxation}. 

\begin{proof}[Proof of Theorem \ref{th:cs-block-relaxation}]

The proof is a variant of \cite[Theorem 9.16]{FoucartRauhut2013}. We start by estimating the probability that the sparse recovery in \eqref{eq:cs-block-relaxation} fails.  According to the optimality criteria \eqref{eq:weighted-cs-condition} for weighted compressed sensing, with the weight $W_\ell$ defined before \eqref{eq:weight-on-support} and the complement $\bar{T}$ of $T$, the probability of failure is bounded by
\begin{multline*}
  P\left( \exists k \in \bar{T}: |\dualp{A X_{\cdot, k}, (A X_{\cdot, T})^{+*} W_{\ell T} \sign(z_T)}| \ge (W_{\ell})_{kk} \right)
  \\
  = P\left( \exists k \in \bar{T}: |\dualp{A^{l(k)} X^{l(k)}_{\cdot, k}, (A X_{\cdot, T})^{+*} W_{\ell T} \sign(z_T)}| \ge (W_{\ell})_{kk} \right),
\end{multline*}
where we have used the block structure of $X$ and $l(k)$ is the number of the block $l$ that contains the index $k \in \{1, \dots, r\theta\}$. With $v := (A X_{\cdot, T})^{+*} W_{\ell T} \sign(z_T)$ we can estimate this by
\begin{multline*}
  P\left( \exists k \in \bar{T}: |\dualp{X^{l(k)}_{\cdot, k}, (A^{l(k)})^* v}| \ge (W_{\ell})_{kk} \right)
  \\
  \begin{aligned}
    & \le P\left( \exists k \in \bar{T}: |\dualp{X_{\cdot, k}^{l(k)}, (A^{l(k)})^* v}| \ge (W_{\ell})_{kk} \; \middle| \; \|v\| \le \alpha \right) + P(\|v\| \ge \alpha).
  \end{aligned}
\end{multline*}
Note that the columns of $X$ involved in $v$ and $X^{l(k)}_{\cdot, k}$ are mutually exclusive, so that these two objects are independent. Therefore, using $(W_\ell)_{kk} = \|X_{\cdot,k}\|_p^p = \|X^{l(k)}_{\cdot,k}\|_p^p$ and $\E[|X^{l(k)}_{j,k}|] = \nu$ for $k \in \bar{T}$ from assumption \eqref{eq:random-X} by Lemma \ref{lemma:inner-product-size}, we have
\begin{equation*}
  P\left(|\dualp{X_{\cdot, k}^{l(k)}, (A^{l(k)})^* v}| \ge (W_{\ell})_{kk} \; \middle| \; \|v\| \le \alpha \right)
  \le
  2 \exp \left( - \frac{\nu^2 n^2}{n + 2 \|(A^{l(k)})^*v\|^2} \right).
\end{equation*}
Since by \eqref{eq:matrix-conditions} we $\|(A^{l(k)})^*v\| \le M(A) \|v\| \le M(A) \alpha$ and we have $R-|T|$ possible choices for $k$, applying a union bound yields
\begin{equation*}
  P\left( \exists k \in \bar{T}: |\dualp{X^{l(k)}_{\cdot, k}, (A^{l(k)})^* v}| \ge (W_{\ell})_{kk} \right)
  \le
  2 (R-|T|) \exp \left( - \frac{\nu^2 n^2}{n + 2 M(A)^2\alpha^2} \right).
\end{equation*}
Finally, estimating $P(\|v\| \ge \alpha)$ by Lemma \ref{lemma:layer-recovery-vector-bound}, we conclude that
\begin{align*}
  P(\text{recovery fail}) 
  & \le 2(R-|T|) \exp \left( - \frac{\nu^2 n^2}{n + 2 M(A)^2\alpha^2} \right) \\
  & \quad + 2 \left( \frac{12}{\delta} \right)^t \exp\left(- c \frac{F_S(A)^2}{M(A)^2} \min \left\{\frac{p_x^2 \delta^2}{4 K^4}, \frac{p_x \delta}{2 K^2} \right\} \right)
\end{align*}
where by the assumption \eqref{eq:delta-alpha} of Lemma \ref{lemma:layer-recovery-vector-bound} the constants are related by 
\begin{equation*}
  1-\delta = \frac{\sqrt{|T|} \overline{s}}{\alpha F_S(A) \sqrt{p_x}},
\end{equation*}
which completes the proof.

\end{proof}

\appendix

\section{Appendix}

\subsection{Probability Limits}
\label{appendix:limits}

\begin{lemma}
  \label{lemma:prob-sequence-limit}
  Assume that $p=p(\gamma) \ge 0$, $q = q(\gamma) > 0$ are two functions and that $\lim_{\gamma \to \infty} p = \lim_{\gamma \to \infty} q = 0$. Then
  \[
    \lim_{\gamma \to \infty} [1-p]^{1/q}
    = \left\{ \begin{array}{ll}
      1 & \text{if }\lim_{\gamma \to \infty} \frac{p}{q} = 0
      \\
      0 & \text{if }\lim_{\gamma \to \infty} \frac{p}{q} = \infty
    \end{array} \right. .
  \]
\end{lemma}

\begin{proof}
 
By $[1-p]^{1/q} = \exp \left( \frac{1}{q} \ln(1-p) \right)$ it is sufficient to compute the limit of the exponent. l'Hospital's rule yields:
\[
  \lim_{\gamma \to \infty} \frac{\ln(1-p)}{q}  
  = \lim_{\gamma \to \infty} \frac{\frac{-1}{1-p} p'}{q'}  
  = \underbrace{\left(\lim_{\gamma \to \infty} \frac{1}{1-p} \right)}_{=1} \left( \lim_{\gamma \to \infty} - \frac{ p'}{q'} \right).
\]
Applying l'Hospital's rule again to the remaining term on the left hand side, we obtain
\[
  \lim_{\gamma \to \infty} \frac{\ln(1-p)}{q}  
  = \lim_{\gamma \to \infty} - \frac{ p'}{q'}
  = \lim_{\gamma \to \infty} - \frac{ p}{q}.
\]
which directly implies the statement of the lemma.

\end{proof}

\begin{lemma}
  \label{lemma:prob-compare}
  Assume that $p=p(\gamma) \ge 0$, $q = q(\gamma) > 0$ are two functions and that
  \begin{align*}
    \lim_{\gamma \to \infty} p & = 0, &
    \lim_{\gamma \to \infty} q & = 0, &
    \lim_{\gamma \to \infty} \frac{p}{q} & = 0.
  \end{align*}
  Then
  \[
    \lim_{\gamma \to \infty} \frac{1- [1-p]^{1/q}}{p/q} = 1
  \]
\end{lemma}

\begin{proof}
  
  By l'Hospital's rule we have
  \begin{align*}
    \lim_{\gamma \to \infty} \frac{1- [1-p]^{1/q}}{p/q}
    & = \lim_{\gamma \to \infty} - \frac{[1-p]^{1/q} \left( \frac{\ln(1-p)}{q} \right)'}{(p/q)'}
    \\
    & = \underbrace{\left( \lim_{\gamma \to \infty} [1-p]^{1/q} \right)}_{=1\text{ by Lemma \eqref{lemma:prob-sequence-limit}}} \left( \lim_{\gamma \to \infty} - \frac{\left( \frac{p}{q} \frac{\ln(1-p)}{p} \right)'}{(p/q)'} \right)
  \end{align*}
  Since $\lim_{\gamma \to \infty} \frac{\ln(1-p)}{p} = -1$ and the assumption $\lim_{\gamma \to \infty} \frac{p}{q} = 0$, we can apply the $\frac{0}{0}$ case of l'Hospital's rule in reverse to the remaining part on the right hand side and obtain

  \begin{equation*}
    \lim_{\gamma \to \infty} \frac{1- [1-p]^{1/q}}{p/q}
    = \lim_{\gamma \to \infty} - \frac{\left( \frac{p}{q} \frac{\ln(1-p)}{p} \right)'}{(p/q)'}
    = \lim_{\gamma \to \infty} - \frac{\frac{p}{q} \frac{\ln(1-p)}{p}}{p/q}
    = 1.
  \end{equation*}
  
\end{proof}

\subsection{Hanson-Wright Inequality}
\label{appendix:hanson-wright}

For the Hanson-Wright Inequality, see e.g. \cite{RudelsonVershynin2013,KasiviswanathanRudelson2019} and the references therein.

\begin{theorem}[{Hanson-Wright Inequality, \cite[Theorem 1.1]{RudelsonVershynin2013}}]
  \label{th:hanson-wright}
  Let $v \in \real^d$ be a vector with independent components with $\E[v_i] = 0$ and $\|v_i\|_{\psi_2} \le K$ and $M \in \real^{d \times d}$ be a matrix. Then, for every $t \ge 0$,
  \[
    \operatorname{Pr} \left[ \left| v^T M v - \E[v^T M v] \right| \ge t \right] \le 2 \exp \left( -c \min \left\{ \frac{t^2}{K^4 \|M\|_F^2}, \frac{t}{K^2 \|M\|} \right\} \right)
  \]
  for a positive absolute constant $c$.
\end{theorem}

For convenience, we restate Corollary \ref{cor:hanson-wright}.

\begin{corollary}
  Let all assumptions of Theorem \ref{th:hanson-wright} be true, and let $C^T C \in \real^{d \times d}$. Then, we have
  \begin{equation*}
    \operatorname{Pr} \left[ \left| v^T C^T C v - \E[v^T C^T C v] \right| \ge \epsilon \|C\|_F^2 \right] 
    \le 2 \exp \left( -c \frac{\|C\|_F^2}{\|C\|^2} \min \left\{ \frac{\epsilon^2}{K^4}, \frac{\epsilon}{K^2} \right\} \right).
  \end{equation*}
\end{corollary}

\begin{proof}

  Setting $M := C^T C$ and $t=\epsilon \|C\|_F^2$ in the Hanson-Wright inequality, we obtain
  \begin{multline*}
    \operatorname{Pr} \left[ \left| v^T C^T C v - \E[v^T C^T C v] \right| \ge \epsilon \|C\|_F^2 \right] 
    \\
    \le 2 \exp \left( -c \min \left\{ \frac{\epsilon^2 \|C\|_F^4}{K^4 \|C^T C\|_F^2}, \frac{\epsilon \|C\|_F^2}{K^2 \|C^T C\|} \right\} \right).
  \end{multline*}
  Thus, using that 
  \begin{align*}
    \|C^T C\|_F^2 & \le \|C^T\|^2 \|C\|_F^2 = \|C\|^2 \|C\|_F^2 \\
    \|C^T C\| & \le \|C\|^2,
  \end{align*}
  we obtain the claimed inequality.
  
\end{proof}

\subsection{Vectorization}
\label{appendix:vectorization}

\begin{lemma}

  \label{lemma:vectorization}
  
  Let $M \in \real^{a \times b}$ and $R \in \real^{b \times c}$ be matrices and $w \in \real^{c}$ a vector. Then
  \begin{enumerate}
  
    \item 
    Identifying the matrix $R \in \real^{b \times c}$ with the tensor $\hat{R} \in \real^b \otimes \real^c$, we have
    \begin{equation}
       M R w = (M \otimes w^T) \hat{R}.
       \label{eq:vectorization-appendix}
    \end{equation}
  
    \item
    If in addition $R$ is a random matrix with i.i.d entries and
    \begin{equation*}
      \begin{aligned}
        \E[r_{ij}]  & = 0, & 
        \E[r_{ij}^2] & = V
      \end{aligned}
    \end{equation*}
    for some $V \ge0$, we have
    \begin{equation*}
      \E \left[ \|M R w\|^2 \right] = V \|M\|_F^2 \|w\|^2.
    \end{equation*}
  \end{enumerate}

\end{lemma}

\begin{proof}

We first identify the matrix $R \in \real^{b \times c}$ with the tensor product $ \hat{R} \in \real^b \otimes \real^c$ via a linear extension of $rs^T \to r \otimes s$. Then, we have
\[
  (M \otimes w^T) (r \otimes s) = Mr \otimes \underbrace{w^T s}_{\in \real} = (Mr) w^T s = M(rs^T)w,
\]
where in the second equality, we have identified $\real^n \otimes \real$ with $\real^n$. By linear extension, we thus have \eqref{eq:vectorization-appendix}.

In order to calculate $\E \left[ \|M R w\|^2 \right]$ note that $\E[\hat{R} \hat{R}^T]_{ij,kl}  = \E[r_{ij} r_{kl}] = V \delta_{ik} \delta_{jl}$ so that
\[
  \E [ \hat{R} \hat{R}^T] = V \, Id,
\]
with identity matrix $Id \in \real^b \otimes \real^b$. It follows that
\begin{align*}
  \E \left[ \|M R w\|^2 \right] 
  & = \E \left[ \|(M \otimes w^T) \hat{R}\|^2 \right] 
  = \E \left[ \hat{R}^T (M^T \otimes w)(M \otimes w^T) \hat{R} \right] 
  \\
  & = \E \left[ \operatorname{tr} \left( \hat{R}^T (M^T M \otimes w^T w) \hat{R} \right) \right] 
  = \E \left[ \operatorname{tr} \left( \hat{R} \hat{R}^T (M^T M \otimes w^T w) \right) \right] 
  \\
  & = \operatorname{tr} \left( \E \left[\hat{R} \hat{R}^T \right] (M^T M \otimes w^T w) \right) 
  = V \operatorname{tr} \left( M^T M \otimes w^T w \right)
  \\
  & = V \|M\|_F^2 \|w\|^2.
\end{align*}

\end{proof}

\subsection{Matrix Norms}

A block matrix $C = \begin{pmatrix} C_1 & \cdots & C_\theta \end{pmatrix}$ has spectral norm
\begin{equation}
  \|C\|^2  \le \sum_{l=1}^{\theta} \|C_l\|^2.
  \label{eq:block-matrix-norms}
\end{equation}
Indeed for any block vector $v = (v_1 \cdots v_\theta)$ we have
\[
  \|Cv\| 
  = \left\| \sum_{l=1}^\theta C_l v_l \right\| 
  \le \sum_{l=1}^\theta \|C_l v_l \|
  \le \sum_{l=1}^\theta \|C_l \| \| v_l \|
  \le \left( \sum_{l=1}^\theta \|C_l \|^2 \right)^{1/2} \left( \sum_{l=1}^\theta \| v_l \|^2 \right)^{1/2}.
\]

\subsection{Weighted Compressed sensing}

This section provides optimality criteria for weighted compressed sensing, analogous  to \cite[Theorems 4.26, 4.30, Corollary 4.28]{FoucartRauhut2013} for the unweighted case. 

\begin{lemma}
  \label{lemma:weighted-cs-optimality-criterium}
  Let $A \in \real^{d \times D}$, $x\in \real^D$, $y \in \real^D$ and $W \in \real^{D \times D}$ be a diagonal weight matrix with non-negative diagonal entries. Let $S$ be the support of $x$ and $\bar{S}$ its complement. If $A_{\cdot, S}$ is injective, $Ax = y$ and 
  \begin{align}
    |\dualp{A_{\cdot, j}, (A_{\cdot, S}^*)^+ W_{\cdot, S} \odot \sign(x_S)}| & < W_{jj}, & j & \in \bar{S},
    \label{eq:weighted-cs-condition}
  \end{align}
  then $x$ is the unique minimizer of the weighted compressed sensing problem
  \begin{align}
    & \min_{x} \|Wx\|_1 = \sum_{i=1}^D w_i |x_i| , & Ax & = y.
    \label{eq:weighted-cs}
  \end{align}

\end{lemma}

\begin{proof}

Define $h = (A_{\cdot, S}^*)^+ W_{\cdot, S} \odot \sign(x_S)$. Since $A_{\cdot,S}^*$ is surjective, we have $A_{\cdot, S}^* h = W_{\cdot, S} \sign(x_S)$ and \eqref{eq:weighted-cs-condition} yields $|\dualp{A_{\cdot, j}, h}| < W_{jj}$ for $j \in \bar{S}$, so that in summary we have
\begin{align*}
  h^* A_{\cdot, j} & = W_{jj} \sign(x_j), & j & \in S
  \\
  h^* A_{\cdot, j} & \in (-W_{jj}, W_{jj}), & j & \in \bar{S},
\end{align*}
which are the KKT conditions for Lagrangian $L = \|Wx\|_1 - h^*(Ax - y)$ of the optimization problem \eqref{eq:weighted-cs} with Lagrange multiplier $h$. Since the optimization problem is convex, the KKT conditions are sufficient, and $x$ is a minimizer, see e.g. \cite[Theorem 3.1.27]{Nesterov2018}. 

To show uniqueness, we consider an elementary proof of this statement. Let $g_{\bar{S}} = A_{\cdot, \bar{S}}^* h$. Then for any $z \in \real^D$ satisfying the constraint $Az=y=Ax$ and using the KKT conditions, we have
\begin{equation*}
  0 
  = \dualp{h, A(z-x)}
  = \dualp{W_{\cdot, S} \sign(x_S), z_S-x_S} + \dualp{g_{\bar{S}}, z_{\bar{S}} - x_{\bar{S}}}.
\end{equation*}
The first term can be estimated by
\begin{multline*}
  \dualp{W_{\cdot, S} \sign(x_S), z_S-x_S}
  = \sum_{j \in S} W_{jj} [\sign(x_j) z_j - \sign(x_j) x_j]
  \\
  \le \sum_{j \in S} W_{jj} [|z_j| - |x_j|]
  = \|W z_S\|_1 - \|W x_S\|_1
\end{multline*}
and using $x_{\bar{S}} = 0$ and the KKT condition, the second by
\[
  \dualp{g_{\bar{S}}, z_{\bar{S}} - x_{\bar{S}}}
  = \dualp{g_{\bar{S}}, z_{\bar{S}}}
  < \|W z_{\bar{S}}\|_1
  = \|W z_{\bar{S}}\|_1 - \|W x_{\bar{S}}\|_1,
\]
with a strict inequality for $z_{\bar{S}} \ne 0$. In conclusion, we have
\[
  0 = \dualp{h, A(z-x)} \le \|W z\|_1 - \|W x\|_1
\]
with with a strict inequality if $z_{\bar{S}} \ne 0$. This shows that $x$ is a minimizer. In case $z_{\bar{S}} = 0$, we have $y = Az = A_S z_S + A_{\bar{S}} z_{\bar{S}} = A_S z_S$ and because $A_S$ is injective $z_S = x_S$. This implies that $z=x$, which shows that $x$ is indeed the unique minimizer.

\end{proof}

\subsection{$NP$-hardness}
\label{sec:np-hard}

It is well known that the $\ell_p$-minimization problem \eqref{eq:cs-p-2} is $NP$-hard in general. For the results of the paper, we consider extra conditions on the sensing matrix $A$ and some constraints on the solution vector $x$. In this section, we show that these conditions do not generally render the problem tractable.

We consider the following three problems. The first two are known to be $NP$-hard and reduced to the compressed sensing problem with additional constraints used in this paper.

\begin{enumerate}

  \item \emph{Exact cover by $3$-set ($X3C_{m,\theta}$):} Given a collection $C^l$, $i=1, \dots, \theta$ of three element subsets of $\{1, \dots, m\}$ does there exits a sub-collection that is a cover of $\{1, \dots, m\}$? I.e. we want to find indices $J \subset \{1, \dots, \theta\}$ such that $\bigcup_{j \in J} C^l = \{1, \dots, m\}$ and $C^l \cap C^k = \emptyset$ for all $l,k \in J$ with $l \ne k$.

  \item \emph{Partition Problem ($PP_m$):} Given: integer or rational numbers $a_1, \dots, a_m$, can one partition $\{1, \dots, m\}$ into two sets $S_1$ and $S_2$ such that $\sum_{i \in S_1} a_i = \sum_{i \in S_2} a_i$?

  \item \emph{$\ell_p$-minimization ($LP_{m,N}^p$):} For $0 \le p \le 1$, given a sensing matrix $A \in \real^{m \times N}$ and measurements $y \in \real^m$, find the minimizer
  \begin{equation*}
    \begin{aligned}
      & \min_{x \in \real^N} \|x\|_p^p, &
      & \text{s.t.}
      & Ax & = y.
    \end{aligned}
  \end{equation*}

\end{enumerate}
For the following discussion, we assume the usual block structure 
\begin{align*}
  A & = \begin{bmatrix} A^1 & \cdots A^\theta \end{bmatrix}, &
  A^l & \in \real^{m \times n}.
\end{align*}
with $N = n\theta$.

We first consider the assumptions in the main result Theorem \ref{th:cs-block-relaxation} on the sensing matrix $A$ or their simplified variants in Section \ref{sec:cs-example}. Since the theorem states a sparse recovery result instead of directly addressing the $\ell_p$-minimization \eqref{eq:cs-p-2}, we consider reductions from the covering problem to $\ell_0$-minimization. For general matrix $A$, the covering problem $X3C_{m,N}$ is polynomial-time reducible to $LP_{m,N}^0$. With the given restrictions on $A$ a reduction is still possible, at least for the smaller problem $X3C_{m,\theta}$. Note however that Theorem \ref{th:cs-block-relaxation} cannot deal with any instance in the following lemma because the solution vector $x$ is contained in the probabilistic part of the statement.

\begin{lemma}
  For $n < m-2$, there is a polynomial-time reduction from $X3C_{m, \theta}$ to $LP_{m+n-1, n\theta}^0$ with blocks of size $A^l \in \real^{m+n-1 \times n}$ that satisfy
  \begin{align*}
    |S^l| & \le \|A^l\|_{\cdot, S^l}^2 \le 3 |S^l|, &
    1 & \le \|A^l\| \le \sqrt{3}
  \end{align*}
  for all index sets $S^l \subset \{1, \dots, n\}$.
\end{lemma}

\begin{proof}

Given an instance of $X3C_{m,\theta}$, let us define the vectors $a^l \in \real^m$ such that $a^l_j = 1$ if $j \in C^l$ and $a^l_j = 0$ else, let $U^l \in \real^{n-1 \times n-1}$ be orthogonal matrices and define the sensing matrix blocks
\[
  A^l = \begin{bmatrix} a^l & \\ & U^l \end{bmatrix} \in \real^{m+n-1 \times n}
\]
and measurement vector
\begin{align*}
  y & = \begin{bmatrix} \bar{y}, \hat{y} \end{bmatrix}, &
  \bar{y} & = \begin{bmatrix} 1 & \cdots & 1\end{bmatrix}^T \in \real^m, &
  \hat{y} & = \begin{bmatrix} 0 & \cdots & 0\end{bmatrix}^T \in \real^{n-1}
\end{align*}
Since all blocks $a^l$ and $U^l$ decouple, the matrix $A$ satisfies all given requirements.

We next show that $X3C_{n, \theta}$ has a solution if and only if the sparsest solution of $Ax = y$ satisfies $\|x\|_0 = m/3$. We first split $x=[x^1, \dots, x^l]$ with $x^l = [v^l, u^l]$ according to the block structure of $A$. This leads to the two decoupled systems
\begin{align*}
  \sum_{l=1}^\theta a^l v^l & = \bar{y}, &
  \sum_{l=1}^\theta U^l u^l & = \hat{y}.
\end{align*}
It directly follows that $u^l = 0$, $l=1, \dots, \theta$. The remaining problem is identical to the original proof in \cite{Natarajan1995} or in the book \cite{FoucartRauhut2013}. Since each column $a^l$ has exactly three non-zero components, we must have $\|x\|_0 \ge m/3$ to obtain a right hand side $\bar{y}$ with all entries one, with equality if and only if there is a cover $J$ and $v^l = 1$ if $l \in J$ and zero else. 

\end{proof}

In this paper, we also consider the case where the solution $x$ comes from a discrete set only. Whereas replacing a continuous variable by a discrete one often makes a problem harder, if we restrict the variables too severely, it might become trivial. With discrete $x$, a reduction from $PP_m$ to $LP_{m+1, 2m}^p$ is particularly simple. Unlike Theorem \ref{th:cs-block-relaxation} this is a $\ell_p$-minimization for $p>0$ so that a direct connection between the theorem and the following lemma can only be made if $A$ allows sparse recovery by $\ell_p$-minimization. Nonetheless, the result indicates that the discrete sets used in Section \ref{sec:cs-example} are not overly simple.

\begin{lemma}
  For $0<p<1$, there is a polynomial-time reduction from $PP_m$ to $LP_{m+1, 2m}^p$ with blocks $A^l \in \real^{m,n}$ with
  \begin{align*}
    |S^l| & \le \|A^l\|_{\cdot, S^l}^2 \le 2 |S^l|, &
    \sqrt{1-\frac{1}{2}} & \le \|A^l\| \le \sqrt{1+\frac{1}{2}}
  \end{align*}
  for any $n$ and even $\theta$ with $n\theta = 2m$, for all index sets $S^l \subset \{1, \dots, n\}$ and solution vector $x$ restricted to $\{-1, -1/2, 0, 1/2, 1\}$.
\end{lemma}

\begin{proof}

The proof is identical to \cite[equation (9)]{GeJiangYe2011}, we only trace the matrix properties. Given an instance of $PP_m$, define the matrix
\begin{align*}
  A & = \begin{bmatrix}
    I & I \\ a^T & - a^T
  \end{bmatrix}, &
  y & = \begin{bmatrix}
    1 & \cdots & 1 & 0
  \end{bmatrix},
\end{align*}
where in the following $I$ denotes the identity matrix of suitable dimensions. Since $\theta$ is even, it follows that each block $A^l$ has the form
\[
  A^l = \begin{bmatrix}
    0 \\
    I \\
    0 \\
    \pm b^T
  \end{bmatrix}
\]
for some vector $b$ that consists of suitable components of $a$. Upon possibly rescaling the last row of $A$, the blocks $A^l$ satisfy all requirements of the lemma.

Let $x$ be a $\ell_p$ minimizer with $Ax=y$. We show that the partition problem has a solution if and only if $\|x\|_p^p = m = n \theta /2$. Let us split the solution as $x = [u,v]$ with $u,v \in \real^{m}$ according to the block structure of $A$. For each component we have $u_i + v_i = 1$ and therefore $|u_i|^p + |v_i|^p \ge 1$ with equality if and only if $u_i=0$ or $v_i=0$. Hence we have $\|x\|_p^p \ge m$ with equality if and only if $u_i=0$ and $v_i=1$ or $u_i=1$ and $v_i=0$ for all $i$, which directly implies the equivalence to the partition problem.

The restriction of $x$ to the given discrete set does not change the argument. Note that the equation $Ax = y$ always has at least the solution $x_i = 1/2$ for all $i$.

\end{proof}

\bibliographystyle{abbrv}
\bibliography{../notes/notes}

\begin{thebibliography}{10}

\bibitem{Allen-ZhuLiSong2019}
Z.~Allen-Zhu, Y.~Li, and Z.~Song.
\newblock A convergence theory for deep learning via over-parameterization.
\newblock In K.~Chaudhuri and R.~Salakhutdinov, editors, {\em Proceedings of
  the 36th International Conference on Machine Learning}, volume~97 of {\em
  Proceedings of Machine Learning Research}, pages 242--252, Long Beach,
  California, USA, 09--15 Jun 2019. PMLR.
\newblock Full version available at \url{https://arxiv.org/abs/1811.03962}.

\bibitem{BaraniukDavenportDeVoreEtAl2008}
R.~Baraniuk, M.~Davenport, R.~DeVore, and M.~Wakin.
\newblock A simple proof of the restricted isometry property for random
  matrices.
\newblock {\em Constructive Approximation}, 28(3):253--263, Dec 2008.

\bibitem{BlumRivest1989}
A.~Blum and R.~L. Rivest.
\newblock Training a 3-node neural network is np-complete.
\newblock In D.~S. Touretzky, editor, {\em Advances in Neural Information
  Processing Systems 1}, pages 494--501. Morgan-Kaufmann, 1989.

\bibitem{BoydVandenberghe2004}
S.~Boyd and L.~Vandenberghe.
\newblock {\em Convex Optimization}.
\newblock Cambridge University Press, New York, NY, USA, 2004.

\bibitem{CandesRombergTao2006a}
E.~J. {Candes}, J.~{Romberg}, and T.~{Tao}.
\newblock Robust uncertainty principles: exact signal reconstruction from
  highly incomplete frequency information.
\newblock {\em IEEE Transactions on Information Theory}, 52(2):489--509, Feb
  2006.

\bibitem{CandesWakinBoyd2008}
E.~J. Cand{\`e}s, M.~B. Wakin, and S.~P. Boyd.
\newblock Enhancing sparsity by reweighted $\ell_1$ minimization.
\newblock {\em Journal of Fourier Analysis and Applications}, 14(5):877--905,
  Dec 2008.

\bibitem{CandesRombergTao2006}
E.~J. Candès, J.~K. Romberg, and T.~Tao.
\newblock Stable signal recovery from incomplete and inaccurate measurements.
\newblock {\em Communications on Pure and Applied Mathematics}, 59, 08 2006.

\bibitem{ChartrandStaneva2008}
R.~Chartrand and V.~Staneva.
\newblock Restricted isometry properties and nonconvex compressive sensing.
\newblock {\em Inverse Problems}, 24(3):035020, may 2008.

\bibitem{ChartrandWotaoYin2008}
R.~{Chartrand} and {Wotao Yin}.
\newblock Iteratively reweighted algorithms for compressive sensing.
\newblock In {\em 2008 IEEE International Conference on Acoustics, Speech and
  Signal Processing}, pages 3869--3872, March 2008.

\bibitem{ConfortiCornuejolsZambelli2014}
M.~Conforti, G.~Cornuéjols, and G.~Zambelli.
\newblock {\em Integer Programming}.
\newblock Springer, 2014.

\bibitem{DaubechiesDeVoreFornasierEtAl2010}
I.~Daubechies, R.~DeVore, M.~Fornasier, and C.~S. Güntürk.
\newblock Iteratively reweighted least squares minimization for sparse
  recovery.
\newblock {\em Communications on Pure and Applied Mathematics}, 63(1):1--38,
  2010.

\bibitem{Donoho2006}
D.~L. {Donoho}.
\newblock Compressed sensing.
\newblock {\em IEEE Transactions on Information Theory}, 52(4):1289--1306,
  April 2006.

\bibitem{DuLeeLiEtAl2019}
S.~Du, J.~Lee, H.~Li, L.~Wang, and X.~Zhai.
\newblock Gradient descent finds global minima of deep neural networks.
\newblock In K.~Chaudhuri and R.~Salakhutdinov, editors, {\em Proceedings of
  the 36th International Conference on Machine Learning}, volume~97 of {\em
  Proceedings of Machine Learning Research}, pages 1675--1685, Long Beach,
  California, USA, 09--15 Jun 2019. PMLR.

\bibitem{FoucartLai2009}
S.~Foucart and M.-J. Lai.
\newblock Sparsest solutions of underdetermined linear systems via
  $\ell_q$-minimization for $0 < q \le 1$.
\newblock {\em Applied and Computational Harmonic Analysis}, 26(3):395 -- 407,
  2009.

\bibitem{FoucartRauhut2013}
S.~Foucart and H.~Rauhut.
\newblock {\em A Mathematical Introduction to Compressive Sensing}.
\newblock Birkh\"{a}user, 2013.

\bibitem{GeJiangYe2011}
D.~Ge, X.~Jiang, and Y.~Ye.
\newblock A note on the complexity of $l_p$ minimization.
\newblock {\em Mathematical Programming}, 129(2):285--299, Oct 2011.

\bibitem{GoodfellowBengioCourville2016}
I.~Goodfellow, Y.~Bengio, and A.~Courville.
\newblock {\em Deep Learning}.
\newblock MIT Press, 2016.
\newblock \url{http://www.deeplearningbook.org}.

\bibitem{GoodfellowVinyals2015}
I.~J. Goodfellow and O.~Vinyals.
\newblock Qualitatively characterizing neural network optimization problems.
\newblock In {\em 3rd International Conference on Learning Representations,
  {ICLR} 2015, San Diego, CA, USA, May 7-9, 2015, Conference Track
  Proceedings}, 2015.

\bibitem{HeZhangRenEtAl2016}
K.~{He}, X.~{Zhang}, S.~{Ren}, and J.~{Sun}.
\newblock Deep residual learning for image recognition.
\newblock In {\em 2016 IEEE Conference on Computer Vision and Pattern
  Recognition (CVPR)}, pages 770--778, June 2016.

\bibitem{KasiviswanathanRudelson2019}
S.~P. Kasiviswanathan and M.~Rudelson.
\newblock Restricted isometry property under high correlations.
\newblock 2019.

\bibitem{LaiXuYin2013}
M.-J. Lai, Y.~Xu, and W.~Yin.
\newblock Improved iteratively reweighted least squares for unconstrained
  smoothed $\ell_q$ minimization.
\newblock {\em SIAM Journal on Numerical Analysis}, 51(2):927--957, 2013.

\bibitem{LiLiang2018}
Y.~Li and Y.~Liang.
\newblock Learning overparameterized neural networks via stochastic gradient
  descent on structured data.
\newblock In S.~Bengio, H.~Wallach, H.~Larochelle, K.~Grauman, N.~Cesa-Bianchi,
  and R.~Garnett, editors, {\em Advances in Neural Information Processing
  Systems 31}, pages 8157--8166. Curran Associates, Inc., 2018.

\bibitem{LorentzGolitschekMakovoz1996}
G.~G. Lorentz, M.~v. Golitschek, and Y.~Makovoz.
\newblock {\em Constructive Approximation: Advanced Problems}.
\newblock Springer-Verlag Berlin Heidelberg, 1996.

\bibitem{Natarajan1995}
B.~K. Natarajan.
\newblock Sparse approximate solutions to linear systems.
\newblock {\em SIAM Journal on Computing}, 24(2):227--234, 1995.

\bibitem{Nesterov2018}
Y.~Nesterov.
\newblock {\em Lectures on Convex Optimization}.
\newblock Springer Publishing Company, Incorporated, 2nd edition, 2018.

\bibitem{RauhutWard2016}
H.~Rauhut and R.~Ward.
\newblock Interpolation via weighted $\ell_1$ minimization.
\newblock {\em Applied and Computational Harmonic Analysis}, 40(2):321 -- 351,
  2016.

\bibitem{RudelsonVershynin2013}
M.~Rudelson and R.~Vershynin.
\newblock Hanson-wright inequality and sub-gaussian concentration.
\newblock {\em Electron. Commun. Probab.}, 18:9 pp., 2013.

\bibitem{SafranShamir2018}
I.~Safran and O.~Shamir.
\newblock Spurious local minima are common in two-layer {R}e{LU} neural
  networks.
\newblock In J.~Dy and A.~Krause, editors, {\em Proceedings of the 35th
  International Conference on Machine Learning}, volume~80 of {\em Proceedings
  of Machine Learning Research}, pages 4433--4441, Stockholmsmässan, Stockholm
  Sweden, 10--15 Jul 2018. PMLR.

\bibitem{ShenLi2012}
Y.~Shen and S.~Li.
\newblock Restricted $p$–isometry property and its application for nonconvex
  compressive sensing.
\newblock {\em Advances in Computational Mathematics}, 37:441--452, 2012.

\bibitem{SoudryCarmon2016}
D.~Soudry and Y.~Carmon.
\newblock No bad local minima: Data independent training error guarantees for
  multilayer neural networks.
\newblock 2016.

\bibitem{Sun2012}
Q.~Sun.
\newblock Recovery of sparsest signals via $\ell_q$-minimization.
\newblock {\em Applied and Computational Harmonic Analysis}, 32(3):329 -- 341,
  2012.

\bibitem{Villani2003}
C.~Villani.
\newblock {\em Topics in Optimal Transportation}.
\newblock American Mathematical Society, 2003.

\bibitem{WoodworthChartrand2016}
J.~Woodworth and R.~Chartrand.
\newblock Compressed sensing recovery via nonconvex shrinkage penalties.
\newblock {\em Inverse Problems}, 32(7):075004, may 2016.

\bibitem{ZhangBengioHardtEtAl2017}
C.~Zhang, S.~Bengio, M.~Hardt, B.~Recht, and O.~Vinyals.
\newblock Understanding deep learning requires rethinking generalization.
\newblock In {\em International Conference on Learning Representations}, 2017.

\end{thebibliography}

\end{document}